\documentclass{article}

\usepackage[margin=1in]{geometry}      
\usepackage{graphicx}
\usepackage{bm}
\usepackage{mathtools}
\usepackage{etoolbox}
\usepackage{amssymb,amsthm}
\usepackage{epstopdf}
\usepackage{enumerate}
\usepackage{dsfont}
\usepackage[linesnumbered, ruled, lined]{algorithm2e}
\usepackage{subcaption}
\usepackage{thmtools}
\usepackage{thm-restate}
\usepackage{tikzit}
\usepackage{caption}
\usepackage{pgfplots}
\pgfplotsset{compat=newest}
\pgfplotsset{every axis legend/.append style={%
cells={anchor=west}}
}
\usetikzlibrary{arrows}
\tikzset{>=stealth'}
\usepackage{hyperref}
\hypersetup{colorlinks=true,citecolor=blue,linkcolor=blue,filecolor=blue,urlcolor=blue,breaklinks=true}

\theoremstyle{plain}
\newtheorem{thm}{Theorem}[section]
\newtheorem{lemma}[thm]{Lemma}
\newtheorem{prop}[thm]{Proposition}

\theoremstyle{definition}

\theoremstyle{remark}
\newtheorem{remark}{Remark}

\newcounter{claimCount}
\AtBeginEnvironment{proof}{\setcounter{claimCount}{0}}

\newcommand{\N}{\mathbb{N}}
\newcommand{\Q}{\mathbb{Q}}
\newcommand{\R}{\mathbb{R}}

\newcommand{\Z}{\mathbb{Z}}
\newcommand{\SOS}{\textsf{SOS}}

\newcommand{\QQ}{\mathbf{Q}}

\newcommand{\E}{\mathbb{E}}

\DeclareMathOperator{\Var}{Var}
\DeclareMathOperator{\Ent}{Ent}

\newcommand{\ideal}[2]{\mathbf{I}_{#2}(#1)}
\newcommand{\abs}[1]{\left\lvert #1\right\rvert}

\newcommand{\norm}[1]{\left\lVert #1\right\rVert}

\DeclarePairedDelimiterX{\inp}[2]{\langle}{\rangle}{#1, #2}

\newcommand{\1}{\mathds{1}}

\newcommand{\inv}{^{-1}}

\newcommand{\st}{\text{ s.t. }}

\DeclareMathOperator{\id}{id}

\DeclareMathOperator{\relint}{relint}

\DeclareMathOperator{\rank}{rank}

\DeclareMathOperator{\linspan}{span}

\newcommand{\cX}{\mathcal{X}}

\renewcommand{\S}{\mathbf{S}} 

\newcommand{\smat}{\texttt{VecToSymMat}}

\newenvironment{sm}{\left(\begin{smallmatrix}}{\end{smallmatrix}\right)}
\newcommand{\lambdamin}{\lambda_{\text{min}}}

\newcommand{\RR}{\mathbb{R}}
\newcommand{\ZZ}{\mathbb{Z}}
\newcommand{\cE}{\mathcal{E}}

\newcommand{\psd}{\succeq}
\newcommand{\pd}{\succ}

\newcommand{\trace}{\text{trace}}
\newcommand{\cA}{{\mathcal A}}

\tikzstyle{node}=[fill=black, draw=black, shape=circle]

\tikzstyle{right}=[->, >=latex]
\tikzstyle{grey edge}=[-, draw={rgb,255: red,128; green,128; blue,128}, fill={rgb,255: red,128; green,128; blue,128}]

\pdfoutput=1

\newcommand{\repourl}{\url{https://github.com/oisinfaust/LogSobolevRelaxations}}
\newcommand{\proofsurl}{\url{https://github.com/oisinfaust/LogSobolevRelaxations/releases/download/v0.1.0/proofs.zip}}

\numberwithin{equation}{section}
\title{Sum-of-Squares proofs of logarithmic Sobolev inequalities on finite Markov chains}
\author{Ois{\'i}n~Faust$^1$ \and Hamza~Fawzi$^1$}

\begin{document}

\maketitle

\begin{abstract}
    Logarithmic Sobolev inequalities are a fundamental class of inequalities that play an important role in information theory. They play a key role in establishing concentration inequalities and in obtaining quantitative estimates on the convergence to equilibrium of Markov processes. More recently, deep links have been established between logarithmic Sobolev inequalities and strong data processing inequalities.    In this paper we study logarithmic Sobolev inequalities from a computational point of view. We describe a hierarchy of semidefinite programming relaxations which give certified lower bounds on the logarithmic Sobolev constant of a finite Markov operator, and we prove that the optimal values of these semidefinite programs converge to the logarithmic Sobolev constant.    Numerical experiments show that these relaxations are often very close to the true constant even for low levels of the hierarchy.    Finally, we exploit our relaxation to obtain a sum-of-squares proof that the logarithmic Sobolev constant is equal to half the Poincar\'e constant for the specific case of a simple random walk on the odd $n$-cycle, with $n\in\{5,7,\dots,21\}$. Previously this was known only for $n=5$ and even $n$.
\end{abstract}


\section{Introduction}\label{sec:intro}

Consider a Markov kernel $K:\cX \times \cX \rightarrow \R_+$ on a finite state space $\cX$, that is, a matrix satisfying
\[
\begin{aligned}
& K_{ij} \geq 0 \quad \forall\, i,j \in \cX\\
& \sum_{j \in \cX} K_{ij} = 1 \quad \forall \,i \in \cX.
\end{aligned}
\]
Together with an initial distribution $\mu_0$ on $\cX$, $K$ defines a continuous-time Markov chain on $\cX$ which evolves according to $\mu_t=P_t\mu_0$, where $P_t$ is the Markov semigroup $P_t = e^{tL}=\sum_{m=0}^\infty \frac{t^m L^m}{m!}$, and $L = K-I$ is the \emph{Laplacian operator}.
 Assuming irreducibility of the Markov kernel, the chain converges as $t\rightarrow \infty$ to the unique stationary distribution $\pi$ satisfying $\pi K = \pi$, and moreover $\pi_i>0\;\forall i\in\cX$. 
Understanding the \emph{mixing time} of this chain, i.e. how fast $\mu_t$ converges to $\pi$, is a subject that has attracted a significant amount of attention in the areas of probability theory, information theory, and dynamical systems. The $L^2$ mixing time can be defined as
\[\tau_2(\epsilon) = \inf\{t>0\st \sup_{\mu_0}\norm{(d\mu_t/d\pi)-1}_{2,\pi}\le \epsilon\}\]
where the supremum is taken over all initial probability distributions $\mu_0$ on $\cX$, and where $\norm{x}_{2,\pi} := \sqrt{\E_\pi[x^2]}$ for $x\in\R^\cX$. By $d\mu_t/d\pi$ we simply mean the density of $\mu_t$ with respect to $\pi$.

\paragraph{The logarithmic Sobolev constant} One way to estimate the mixing time $\tau_2$ is via \emph{logarithmic Sobolev inequalities} \cite{P.Diaconis1996}, which we recall now.
For $x\in\R_+^\cX$, $x\neq0$, we define
\[\Ent_\pi x= \E_\pi[x\log x] - \E_\pi [x]\log\E_\pi[x]= \sum_{i\in\cX}\pi_ix_i\log\left(\frac{x_i}{\E_\pi[x]}\right),\]
where $\E_\pi[x]=\sum_{i\in\cX}\pi_ix_i$. This quantity measures discrepancy between $\pi$ and the probability distribution $\sigma$ whose density with respect to $\pi$ is $x$. 
In fact, when $\E_\pi[x]=1$, $\Ent_\pi x = D(\sigma\Vert \pi)$ is the relative entropy between $\pi$ and $\sigma_i=\pi_ix_i$. For $x,y\in\R^\cX$, we define the \emph{Dirichlet form} corresponding to $K$
\[ \mathcal{E}_K(x,y) = -\inp{x}{Ly}_\pi = -\sum_{i,j}\pi_i x_i(Ly)_i,\]
in terms of the inner product $\inp{x}{y}_\pi=\E_\pi[xy]=\sum_{i\in\cX}\pi_i x_i y_i$.
When it is clear which kernel $K$ is under consideration we may simply write $\mathcal{E}$.
The Dirichlet form is bilinear, and the quadratic form it induces, 
\[\mathcal{E}(x,x) = -\inp{x}{Lx}_\pi = \frac{1}{2}\sum_{i,j}\pi_iK_{ij}(x_i-x_j)^2\]
is positive semidefinite. 
Since $K$ is irreducible, $\mathcal{E}(x,x)$ vanishes if and only if $x$ is constant. A \emph{logarithmic Sobolev inequality} has the form
\begin{equation}\label{eq:lsineq}
    \cE(x,x) \ge \alpha\Ent_\pi[x^2]\qquad\forall\,x\in\R^\cX \ ;
\end{equation}
the \emph{logarithmic Sobolev constant} of the semigroup of the Markov kernel $K$ is the largest constant $\alpha$ satisfying \eqref{eq:lsineq}.
The $L^2$ mixing time of the chain can then be bounded in terms of $\alpha$ as follows \cite[Theorem 3.7]{P.Diaconis1996}
\begin{equation}
    \label{eq:alphabound}
    \tau_2(1/e) \leq \frac{1}{\alpha}\left(1+\frac{1}{2}\log\log(\pi_*^{-1})\right)
\end{equation}
where $\pi_* = \min_{i \in \cX} \pi_i$. 

\paragraph{Poincar\'e inequality} A similar, but older, tool than the logarithmic Sobolev inequality is the \emph{Poincar\'e inequality}
\begin{equation}
    \label{eq:poincareineq}
    \cE(x,x) \ge \lambda\Var_\pi[x] \qquad\forall\,x\in\R^\cX \ , 
\end{equation}
where $\Var_\pi[x] := \E_\pi[x^2] - \E_\pi[x]^2$. The \emph{Poincar\'e constant} of the Markov kernel $K$ is the largest constant $\lambda$ satisfying \eqref{eq:poincareineq}.
A classical result is that $\lambda\ge2\alpha$, see e.g., \cite[Lemma 3.1]{P.Diaconis1996}.
The Poincar\'e constant is sometimes known as the \emph{spectral gap} and, if the semigroup is reversible, it is equal to the second smallest eigenvalue of $-L$ (note that the smallest eigenvalue of $-L$ is 0).
Using the Poincar\'e constant, one can obtain the following bound on the mixing time \cite[Corollary 1.6]{Montenegro2006}
\begin{equation}
\label{eq:poincarebound}
\tau_2(1/e) \leq \frac{1}{\lambda}\left(1+\frac{1}{2}\log(\pi_*^{-1})\right).
\end{equation}

The main advantage of the bound \eqref{eq:alphabound} compared to \eqref{eq:poincarebound} is the dependence on $\pi_*$ which is exponentially smaller in \eqref{eq:alphabound} vs. \eqref{eq:poincarebound}. 
So if $\frac{\alpha}{\lambda}\gg \frac{2+\log\log(\pi_*^{-1})}{2+\log(\pi_*^{-1})}$, the bound based on the logarithmic Sobolev constant will be much better than the one based on the Poincar\'e constant. A typical example is the simple random walk on the hypercube $\cX=\{-1,+1\}^n$ where $\lambda=2/n$ and $\alpha=1/n$, while $\pi_*^{-1} = 2^n$. So for this example, the $L^2$ mixing time bound based on $\alpha$ is better than the bound based on $\lambda$ by a factor of approximately $n/\log_2{n}$.

\paragraph{Strong data processing inequalities}
Logarithmic Sobolev inequalities arise not only in the study of mixing times of Markov chains, but are also connected to the study of strong data processing inequalities for memoryless channels. It can be shown (see \cite{Raginsky13, Raginsky16}) that a strong data processing inequality is implied by a logarithmic Sobolev inequality on a certain Markov kernel derived from the channel. In Section \ref{sec:sdpi} we give more details on this connection, and sketch how our methods can be modified to  prove strong data processing inequalities directly.

\subsection{Contributions}

In this paper we study the problem of \emph{computing} the logarithmic Sobolev constant $\alpha$ of a given Markov kernel $K$. Whereas computing the Poincar\'e constant of $K$ reduces to an eigenvalue problem which can be solved using standard algorithms, computing the constant $\alpha$ seems much harder. One can read in Saloff-Coste's lecture notes on finite Markov chains \cite[Section 2.2.2]{Saloff-Coste1997}: ``\textit{The natural question now is: can one compute or estimate the constant $\alpha$? Unfortunately, the present answer is that it seems to be a very difficult problem to estimate $\alpha$}''. Indeed $\alpha$ is defined in terms of a nonconvex optimization problem, and finding the exact constant $\alpha$ for small chains can require very lengthy and technical computations, see e.g., \cite{Chen2008}.

In this work we propose a new method to compute accurate and certified lower bounds on the logarithmic Sobolev constant of a Markov chain $K$. We make use of the powerful \emph{sum of squares} framework for global optimisation. Whereas sum of squares techniques are most directly applied to \emph{polynomial} inequalities, we show how this framework can be brought to bear on entropy-based functional inequalities by means of rational Pad\'e approximants to the logarithm.
For integers $d\ge2$, we describe a semidefinite program whose solution $\alpha_d$ is a lower bound on $\alpha$. For fixed $d$ the size of the semidefinite program grows polynomially in $\abs{\cX}$. The main technical result of the paper can be summarized in the following:

\begin{thm}\label{thm:hierarchy-main-intro}[see detailed version in Theorem \ref{thm:hierarchy-main}]
    Let $\mathcal{E},\,\pi$ be the Dirichlet form and stationary distribution of an irreducible Markov chain on $\cX=\{1,\dots,n\}$, with log-Sobolev constant $\alpha$. For each $d \geq 2$, there is a number $\alpha_d \leq \alpha$ such that the following is true:\\
    (i) $\alpha_d$ can be computed with a semidefinite program of size $O(n^{2d})$\\
    (ii) $\alpha_{d} \uparrow \alpha$ as $d\uparrow\infty$.
\end{thm}

\subsection{Overview of approach} Our approach relies on semidefinite programming, and uses the sum of squares paradigm to prove inequalities. Sum of squares programming is a general method to search for proofs of polynomial inequalities \cite{Lasserre2001,parrilo2003semidefinite}. The sum of squares method cannot, however, be used directly on \eqref{eq:lsineq} as the right-hand side involves a nonpolynomial function. To circumvent this, we search for stronger inequalities where the term $\Ent_\pi[x^2]$ is upper bounded by a polynomial function. Some simple upper bounds based on Taylor expansions can be used and already yield nontrivial bounds. However because of the poor approximation properties of Taylor expansions (finite radius of convergence), these methods are not convergent in general and cannot yield the kind of guarantee of Theorem \ref{thm:hierarchy-main-intro}. The main novelty in our work is that we consider instead rational Pad\'e approximants which converge pointwise to $\log$ everywhere on the positive line, with an exponential convergence rate on any compact set. This choice of rational bounds, instead of Taylor expansions, is key to obtain the convergence result of Theorem \ref{thm:hierarchy-main-intro}.

In addition to the stated convergence result, numerical experiments (see Section \ref{sec:numerics}) show that these relaxations are often very close to the true constant, even for low levels of the hierarchy. In fact, as shown in Section \ref{sec:cycle} for the $n$-cycle, the relaxation can even yield the exact value $\alpha$ already for small values of $d$. If the entries of the kernel $K$ are rational, we further explain how to turn the output of the semidefinite program into a formally verified bound on $\alpha$ with rational arithmetic.
An implementation of our method is available in the Julia language, with examples at:
\[
\text{\repourl}.
\]

Importantly, the approach that we develop here is versatile enough to be applied for other entropic functional inequalities. Most notably, we explain in Section \ref{sec:mlsi} how similar techniques can be used to obtain bounds on the \emph{modified logarithmic Sobolev constant}, and the \emph{strong data processing constant} of a channel. (The latter application appeared in the conference paper \cite{sdpi-isit}.)

\subsection{Related work} As far as we are aware this is the first work that studies the logarithmic Sobolev constant from a computational point of view. 
In theoretical computer science, the papers \cite{barak2012hypercontractivity,kauers2017hypercontractive} show that certain \emph{hypercontractive inequalities}, associated to a particular Markov kernel on the hypercube $\{-1,1\}^n$, can be proved using sums of squares. It turns out that hypercontractive inequalities are strongly related to logarithmic Sobolev inequalities. A hypercontractive inequality associated to a Markov semigroup $P_t$ takes the form
\begin{equation}
\label{eq:hypercontractiveineq}
\|P_t x\|_{q,\pi} \leq \|x\|_{p,\pi}
\end{equation}
where $q = q(t,p) > p$, and we use the notation $\|x\|_{p,\pi} = \left(\sum_{i \in \cX} \pi_i |x_i|^p\right)^{1/p}$. Proving a family of inequalities of the form \eqref{eq:hypercontractiveineq} can be done by proving a single logarithmic Sobolev inequality of the form \eqref{eq:lsineq}. 
In fact, if $P_t$ is reversible, then \eqref{eq:hypercontractiveineq} is \emph{equivalent} to a logarithmic Sobolev inequality with $q=1+(p-1)e^{4\alpha t}$.
This was one of the first motivations for the introduction of logarithmic Sobolev inequalities by Gross in \cite{Gross1975}. We refer to \cite[§5.2.2]{bakry2013analysis} for more details on this connection.

Related to our work also are the papers \cite{boyd2004fastest,sun2006fastest} which formulate the problem of finding the Markov chain with the largest Poincar\'e constant as a convex optimization problem. In fact the paper \cite{boyd2004fastest} concludes by considering the problem of finding the Markov chain with the largest logarithmic Sobolev constant. However no algorithm is proposed as there was no known way of computing/bounding $\alpha$.
The bound we propose in this method can, in fact, be used for this matter. We comment on this application in the discussion section (Section \ref{sec:discussion}).


\subsection{Organization} 
The rest of this paper is organized as follows. 
In Section \ref{sec:review} we review some preliminary results concerning semidefinite programming and sums of squares. 
In Section \ref{sec:approach} we present our approach to bounding the logarithmic Sobolev constant, and we prove the convergence of the hierarchy of semidefinite programs (Theorem \ref{thm:hierarchy-main-intro}). We also discuss some practical aspects related to these bounds, and show how the solution of the semidefinite programs can be turned into formal certificates using exact arithmetic.
In Section \ref{sec:mlsi} we explain how the techniques used for the log-Sobolev constant can be applied to other entropic functional inequalities, namely the modified logarithmic Sobolev constant, and the strong data processing constant.
Finally, in Sections \ref{sec:numerics} and \ref{sec:cycle} we illustrate our method on various examples of Markov chains; in particular we use our method to obtain the exact log-Sobolev constant of the simple random walk on the $n$-cycle for $n\in\{5,7,9,\dots,21\}$.

\section{Background}\label{sec:review}
\label{sec:sdp_sos}

\label{sec:sos_notation}

Given a variable $t$, we let $\R[t]$ (resp. $\R[t]_d$) be the vector space of univariate real polynomials (resp. of degree at most $d$). 
Given variables $x_1,x_2,\dots,x_n$, $\R[x]=\R[x_1,\dots,x_n]$ denotes the space of polynomials with real coefficients in $x=(x_1,x_2,\dots,x_n)$, and $\R[x]_d$ is the (finite-dimensional) subspace of $\R[x]$ containing only polynomials of degree at most $d$.

\paragraph{Sums of squares} A polynomial $p \in \R[x]$ is a \emph{sum of squares} if we can find polynomials $p_1,\ldots,p_m$ such that $p = \sum_{i} p_i^2$. The set of polynomials that can be written as a sum of squares is a convex cone inside $\R[x]$ which we will denote by $\Sigma[x]$. 
Given a subspace $V\subseteq\R[x]$, we define the cone of polynomials that can be written as a sum of squares of polynomials from $V$ as:
\begin{equation}
\label{eq:SigmaV}
\Sigma(V) = \left\{ \sum_{i} q_i^2 : q_i \in V \right\} = \operatorname{cone}(\{q^2 : q \in V\}).
\end{equation}
Of particular importance is the case $V=\R[x]_d$ which corresponds to polynomials that can be written as a sum of squares of polynomials of degree at most $d$. We reserve the notation $\Sigma[x]_{2d}:=\Sigma(\R[x]_d)$ for this set, which has the property that $\Sigma[x]_{2d}=\Sigma[x]\cap\R[x]_{2d}$.

A fundamental fact about $\Sigma(V)$, is that one can decide membership in it using \emph{semidefinite programming} \cite{Shor1987ClassOG,Parrilo2000, Lasserre2001}. 
This is because a polynomial $f \in \R[x]$ belongs to $\Sigma(V)$ if and only if there exists a positive semidefinite \emph{Gram matrix} $Q$ satisfying $f=b(x)^{\top} Q b(x)$ where $b(x) = (b_1(x),\ldots,b_{\dim V}(x))$ is a basis of $V$. Indeed, a positive semidefinite $Q \psd 0$ can be equivalently written as $Q = \sum_{i} q_i q_i^{\top}$, and thus
\begin{equation}
\label{eq:sosV}
f(x) = b(x)^{\top} Q b(x), \;\; Q \psd 0 \iff f(x) = \sum_{i} (q_i^{\top} b_i(x))^2
\end{equation}
is a sum of squares of polynomials in $V$.

\paragraph{Semidefinite programming} Recall that a semidefinite program (SDP) is an optimization problem of the form
\begin{equation}
    \label{eq:sdp}
    \min_{X \in \S^n} \trace(CX) \;\;\; \text{s.t.} \;\;\; \cA(X) = b, X \psd 0
\end{equation}
where $\S^n$ is the space of $n\times n$ real symmetric matrices, $X \psd 0$ means that $X$ is positive semidefinite; and $C \in \S^n$, $\cA:\S^n\to \R^m$, $b \in \R^m$ are given. The set of positive semidefinite matrices is denoted $\S^n_+$. Observe that, by \eqref{eq:sosV}, $\Sigma(V)$ can be described as the feasible set of a semidefinite program
\[
\Sigma(V) = \{f \in \RR[x] : \exists Q \psd 0, \; \sum_{i,j=1}^{\dim V} Q_{ij} b_i(x) b_j(x) = f(x)\}.
\]
(Note that $\sum_{i,j=1}^{\dim V} Q_{ij} b_i(x) b_j(x) = f(x)$ forms a set of linear equality constraints on $Q$.)

Whilst it is generally NP-hard to decide whether a polynomial takes only nonnegative values on $\R^n$ \cite{Murty87}, there are efficient algorithms for solving SDPs to any desired precision using floating point arithmetic.
Therefore, for many computational tasks involving constraints specifying that certain polynomials be nonnegative, we can construct relaxations based on the more tractable condition that these polynomials are sums of squares. 

\paragraph{Constrained polynomial optimization}
In many situations, we are interested in certifying nonnegativity of a polynomial $f$ on a particular subset $S \subset \RR^n$. If $S$ is described by a finite set of polynomial equations and inequalities
\begin{equation}
\label{eq:semialgebraic}
S = \{x \in \RR^n : g_1(x) \geq 0,\ldots,g_J(x) \geq 0, \; h_1(x) = 0 \; h_K(x) = 0\}
\end{equation}
then an obvious \emph{sufficient} condition for $f(x) \in \RR[x]$ to be nonnegative on $S$ is that
\begin{equation}
\label{eq:pop}
f(x) = \sigma_0(x) + \sigma_1(x) g_1(x) + \dots + \sigma_J(x) + \phi_1(x) h_1(x) + \dots + \phi_K(x) h_K(x)
\end{equation}
where $\sigma_0,\ldots,\sigma_J$ are sums of squares, and $\phi_1,\ldots,\phi_K$ are arbitrary polynomials. If we restrict the degrees of $(\sigma_j)$ and $(\phi_k)$ by $\deg(\sigma_0) \leq 2d$, $\deg (\sigma_j g_j) \leq 2d$ and $\deg(\phi_k h_k) \leq 2d$, then the condition \eqref{eq:pop} can be expressed as a semidefinite feasibility problem having one semidefinite constraint of size $\dim \R[x]_{d} \times \dim \R[x]_{d}$ and $J$ semidefinite constraints of size $\dim \R[x]_{d - \frac{1}{2}\deg g_j} \times \dim \R[x]_{d-\frac{1}{2}\deg g_j}$ each \cite{parrilo2003semidefinite}.

\paragraph{Univariate polynomials}\label{sec:univariate}
For univariate polynomials, global nonnegativity is equivalent to being a sum of squares. See for example \cite[Theorem 2.3]{lasserre_2015}. We record the following result which will be useful later:
\begin{thm}[e.g. Theorem 2.4 in \cite{lasserre_2015}]\label{thm:char-uni-nonneg}
    Let $p\in\R[t]$ have degree $k$, and suppose $p(t)\ge0$ for all $t\in[-1,1]$.
    \begin{itemize}
        \item If $k$ is odd, then
        $p \in (1-t)\Sigma[t]_{k-1}+(1+t)\Sigma[t]_{k-1}$.
        \item If $k$ is even, then
        $p \in \Sigma[t]_{k}+(1-t^2)\Sigma[t]_{k-2}$.
    \end{itemize}
\end{thm}

\section{Our approach}\label{sec:approach}

Consider a Markov chain on $\mathcal{X}=\{1, \dots, n\}$ with irreducible transition matrix $K$, (unique) stationary distribution $\pi$, and Dirichlet form
\[ \mathcal{E}(x,x) = \frac{1}{2}\sum_{i,j \in \cX}\pi_iK_{ij}(x_i-x_j)^2. \]
Denote $\pi_*=\min\{\pi_i\mid i\in \cX\}$. Since the kernel $K$ is irreducible and $\cX$ is finite, $\pi_*>0$.

We are interested in a systematic way of determining numbers $\gamma\in \R$ such that
\begin{equation}\label{cheese}
\gamma \mathcal{E}(x,x) \ge \Ent_\pi[x^2] := \sum_{i\in\cX} \pi_i x_i^2 \log\left(\frac{x_i^2}{\norm{x}_\pi^2}\right)
\end{equation}
holds for every $ x\in \R^n$. Then $\gamma\inv$ will be a lower bound on $\alpha$, the logarithmic Sobolev constant of the chain.

Since \eqref{cheese} is homogeneous in $x$, we need only prove it for $x$ satisfying $\norm{x}_\pi^2=1$. 
We may also assume that the components of $x$ are nonnegative, since $\mathcal{E}(\abs{x}, \abs{x}) \le \mathcal{E}(x,x)$.
This simplifies the right hand side to $\sum_{i\in\cX}2\pi_ix_i^2\log x_i$.
Let us write $S^\pi$ for the intersection of the hyperellipsoid $\norm{x}_\pi^2=1$ with the nonnegative orthant 
\[ S^\pi = \{x\in\R_+^n\mid\norm{x}_\pi^2=1\}.\]
Now the inequalities we are interested in have the form
\begin{equation}\label{bigcheese}
    \gamma \mathcal{E}(x, x) \ge \sum_{i\in\cX} 2\pi_i x_i^2 \log\left(x_i\right)
\end{equation}
whenever $x\in S^\pi$. The right hand side of \eqref{bigcheese} is  not a polynomial in $x$. In order to make use of the sum of squares machinery described in Section \ref{sec:sdp_sos}, we will replace the right hand side of \eqref{bigcheese} by an upper bound which is a polynomial.
Specifically, if $P(x)\in\R[x]$ is a polynomial satisfying $P(x)\ge \Ent_\pi[x^2]$ for all $x\in S^\pi$, then the polynomial inequality
\begin{equation}\label{mediumcheese}
    \gamma \mathcal{E}(x,x) \ge P(x) \qquad\forall\;x\in S^\pi
\end{equation}
implies \eqref{bigcheese}. Fixing $d\geq \frac{1}{2}\deg P$, we can then find valid values of $\gamma$ by solving the following sum of squares program:
\begin{equation}\label{soscheese}
    \min_{\gamma\in \R}\gamma \;\st \; \gamma\mathcal{E}(x,x) - P(x) \in \SOS_d(S^{\pi})
\end{equation}
where $\SOS_d(S^{\pi})$ is a degree-$d$ sum-of-squares relaxation of the set of nonnegative polynomials on $S^{\pi}$ following \eqref{eq:pop}, namely:
\begin{equation}
\label{eq:SOSOpi}
\begin{aligned}
\SOS_d(S^{\pi}) := \Biggl\{ f \in \RR[x]_{2d} : & \; f(x) = \sigma_0(x) + \sum_{i=1}^n x_i \sigma_i(x) + \phi(x) (\|x\|_{\pi}^2-1)\\
& \qquad\qquad \text{ where } \sigma_0 \in \Sigma[x]_{2d}, \; \sigma_i \in \Sigma[x]_{2d-2} \;\;(i=1,\ldots,n)\\
& \qquad \qquad \qquad \qquad \phi \in \RR[x]_{2d-2} \Biggr\}.
\end{aligned}
\end{equation}
Two questions should be addressed:
\begin{itemize}
\item How to choose the polynomial $P$ which approximates $\Ent_\pi[x^2]$ on $S^\pi$?
\item Having chosen a polynomial $P$ such that $\Ent_{\pi}[x^2] \le P(x)\le(1+\epsilon)\Ent_\pi[x^2]$ for all $x \in S^{\pi}$, can we guarantee that the sum-of-squares relaxation \eqref{soscheese} will be $\epsilon$-close to the log-Sobolev constant $\alpha$, for large enough $d$?
\end{itemize}
We address these two questions in the following two subsections, starting by the latter.

\subsection{The sum-of-squares relaxation \eqref{soscheese}}

Assume we have a polynomial $P(x)$ which is an $\epsilon$-approximation to the entropy function $\Ent_{\pi}[x^2]$ on $S^{\pi}$. In this subsection we show that the sum-of-squares program \eqref{soscheese} will give, for large enough $d$, an $\epsilon$-approximation of the logarithmic Sobolev constant. This is the object of the next theorem.

\begin{thm}\label{hierarchy-convergence}
    Let $\mathcal{E},\,\pi$ be the Dirichlet form and stationary distribution of an irreducible Markov chain on $\cX=\{1,\dots,n\}$ with log-Sobolev constant $\alpha$.
    Let $\epsilon\in(0,1)$, and suppose that $P(x)\in\R[x]$ is a polynomial such that, for all $x\in S^{\pi}$, we have
    \begin{equation}\label{eq::poly-approximates-eps}
        \Ent_{\pi}[x^2] \le P(x)\le(1+\epsilon)\Ent_\pi[x^2].
    \end{equation}
    Let $\gamma^*_d$ be the optimal value of \eqref{soscheese}, and call $\alpha^P_d = 1/\gamma^*_d$. Then $\alpha^P_d \leq \alpha$. Furthermore, there exists $d_0\in\N$ such that for every $d\ge d_0$, $\alpha^P_d \ge (1-\epsilon)\,\alpha$.
\end{thm}

It is a foundational result in sum-of-squares programming that any polynomial which is \emph{strictly positive} on a compact basic semialgebraic set (i.e., a compact set of the form \eqref{eq:semialgebraic}) has a sum-of-squares representation on that set, provided an additional so-called Archimedean condition is satisfied. This is known as Putinar's Positivstellensatz \cite{Putinar93}.
We cannot directly use this result however, as the assumption \eqref{eq::poly-approximates-eps} on $P$ implies that $P(\1) = 0$ and so $\gamma\mathcal{E}(\1,\1) - P(\1) = 0$.
Instead we will use a result from a more recent line of work \cite{Scheiderer03, Marshall2006, Nie14}, which guarantees the existence of a sum-of-squares representation provided the second-order sufficient conditions of optimality are satisfied at the vanishing points of the polynomial.  The most convenient result for our purposes is \cite[Theorem 1.1]{Nie14} which we cite here in the particular case of $S^{\pi}$ for convenience.

\begin{thm}[{\cite[Theorem 1.1]{Nie14}}]\label{thm:nie}
Assume $f(x) \in \R[x]$ is a polynomial such that $f(x) \geq 0$ for all $x \in S^{\pi}$. Let $h(x) = \|x\|_{\pi}^2 - 1$, and assume that any zero $x^*$ of $f$ in $S^{\pi}$ satisfies the following:
\begin{itemize}
\item $x_i^* > 0$ for each $i \in \{1,\ldots,n\}$
\item $\nabla (f - \kappa h)(x^*) = 0$ for some $\kappa \in \RR$ (which can depend on $x^*$)
\item $\nabla^2 (f - \kappa h)(x^*)$ is positive definite on the subspace $\{v \in \R^n : \sum_{i=1}^n \pi_i x_i^* v_i = 0\}$ (the tangent space to $S^{\pi}$ at $x^*$).
\end{itemize}
Then there exists large enough $d$ such that $f \in \SOS_{d}(S^{\pi})$.
\end{thm}

Equipped with this theorem, we are ready to prove Theorem \ref{hierarchy-convergence}.

\begin{proof}[Proof of Theorem \ref{hierarchy-convergence}]
The fact that $\alpha^P_d \leq \alpha$ is obvious by construction. We thus focus on the second claim of the theorem.

Let $f_1(x):=\frac{1}{(1-\epsilon)\,\alpha}\,\mathcal{E}(x,x) - P(x)$. Our goal is to show that for large enough $d$, $f_1 \in \SOS_d(S^{\pi})$. 
By definition of the log-Sobolev constant and \eqref{eq::poly-approximates-eps}, the polynomial $f_0(x):=\left(\frac{1+\epsilon}{\alpha}\right)\mathcal{E}(x,x) - P(x)$ is nonnegative on $S^\pi$.
        Note that $1+\epsilon < (1-\epsilon)^{-1}$ and, by irreducibility, $\mathcal{E}(x,x)$ is strictly positive on $S^\pi\backslash\{\1\}$ (this follows from Perron-Frobenius).
        Since
        \[
        f_1(x) = f_0(x) + \frac{2\epsilon}{(1-\epsilon^2)\alpha}\cE(x,x),
        \]
the polynomial $f_1$ is positive on $S^\pi\backslash\{\1\}$, and zero at $x=\1$.

    It remains to check that the conditions of Theorem \ref{thm:nie} hold for $f_1$ at $x^*=\1$. Certainly $\1>0$, and since $x^*=\1$ is a minimum of $f_1$ on $S^{\pi}$, the first-order optimality condition also holds for some $\kappa\in\R$.

    Let $L_0(x) = f_0(x) - \kappa(\norm{x}_\pi^2-1)$, and $L_1(x) = f_1(x) - \kappa(\norm{x}_\pi^2-1)$.
    The second-order sufficiency condition for $f_1$, which it remains to verify, is that for any $v\neq0\st\inp{v}{\1}_\pi=0$, $v^\top\nabla^2L_1(\1)v>0$.
    By the second-order \emph{necessary} condition for $f_0$ to be nonnegative, $v^\top\nabla^2L_0(\1)v\ge0$ for any such $v$, so it suffices to check that $v^\top\nabla^2\cE(\1,\1)v>0$ (since $L_1(x)=L_0(x) + \frac{2\epsilon}{(1-\epsilon^2)\alpha}\cE(x,x)$). We have $v^\top\nabla^2\cE(\1,\1)v=2\cE(v,v)$, and by irreducibility of the Markov kernel, this quantity is nonnegative can only vanish if $v\in\R\1$ which is impossible since $v\neq0$ and $\inp{v}{\1}_\pi=0$.
\end{proof}

\begin{remark}[Rate of convergence]
It would be satisfying to prove an upper bound on $d_0$ in terms of $n$, or to characterize the rate of convergence of the solutions $\gamma$ to $\sup_{x\in S^\pi}\frac{P(x)}{\cE(x,x)}$, as was done e.g., in \cite{Reznick1995,Doherty2012,Fang2019}.
Unfortunately, as far as we are aware such results are only available for the strictly positive case, and even then, the resulting degree bounds are too large to implement in practice. Our experiments in Section \ref{sec:numerics} suggest that finite convergence often occurs already at $d=3$, at least for small chains.
\end{remark}

Whilst Theorem \ref{hierarchy-convergence} asserts that for large enough $d$, $\alpha^P_d \geq (1-\epsilon) \alpha$, it is not a priori obvious that for small values of $d$, the value of the relaxation $\alpha^P_d$ is even strictly positive, i.e., that \eqref{soscheese} is feasible.
We show below, in Proposition \ref{feas}, that under some mild conditions on $P$ the program \eqref{soscheese} is indeed feasible for any $d$, and hence can already yield nontrivial bounds on $\alpha$ for small values of $d$.

\begin{prop}\label{feas}
Let $\mathcal{E},\,\pi$ be the Dirichlet form and stationary distribution of an irreducible Markov chain on $\cX=\{1,\dots,n\}$ with log-Sobolev constant $\alpha$.
    Suppose that $P(x)\in\R[x]$ is a polynomial such that, for all $x\in S^{\pi}$, we have
    \begin{equation}\label{eq::poly-approximates-eps-b}
        \Ent_{\pi}[x^2] \le P(x)\le(1+\epsilon)\Ent_\pi[x^2]
    \end{equation}
    for some $\epsilon > 0$.
Then for any $d \geq (\deg P)/2$, $\alpha^P_d > 0$, i.e., \eqref{soscheese} is feasible.
\end{prop}
\begin{proof}
We only include a brief sketch, and we defer the technical details to the appendix.
First we note that, from \eqref{eq::poly-approximates-eps-b}, $P(x) \geq 0$ for $x \in S^{\pi}$ and furthermore $P(\1) = 0$. Thus, by the first-order conditions of optimality we must have $\nabla P(\1) = 2\kappa \pi$ for some $\kappa \in \RR$, or equivalently $\nabla L(\1) = 0$ where 
\[
L(x) = P(x) - \kappa(\|x\|_{\pi}^2-1).
\]
Let $b_1(x),\ldots,b_N(x)$ be a basis for the subspace of polynomials $\{p \in \RR[x]_{d} : p(\1) = 0\}$. By Lemma \ref{lem:spanbibj} in the Appendix we know that
\[
\linspan \left\{ b_i b_j : 1 \leq i,j \leq N \} = \{p \in \RR[x]_{2d} : p(\1) = 0 \text{ and } \nabla p(\1) = 0\right\}.
\]
This means that we can write 
\begin{equation}
\label{eq:LPkappaB}
L(x) = P(x) - \kappa(\|x\|_{\pi}^2-1) = b(x)^T B b(x)
\end{equation}
for some $N\times N$ symmetric matrix $B$, where $b(x) = [b_i(x)]_{1\leq i \leq N}$.

To complete the proof, we will show there exists a positive \emph{definite} matrix $A \pd 0$ such that
\begin{equation}
\label{eq:EbAb}
\cE(x,x) = b(x)^T A b(x) \bmod (\|x\|_{\pi}^2 - 1).
\end{equation}
This will conclude the proof, because then by choosing $\gamma$ large enough such that $\gamma A - B \psd 0$, we get
\[
\gamma \cE(x,x) - P(x) = b(x)^T (\gamma A - B) b(x) \bmod (\|x\|_{\pi}^2 - 1).
\]
Since $\gamma A - B \psd 0$, the term $b(x)^T (\gamma A - B) b(x)$ is a sum-of-squares of polynomials, and so this establishes that $\gamma \cE(x,x) - P(x) \in \SOS_d(S^{\pi})$ which is what we wanted.

We focus on proving \eqref{eq:EbAb}. From the Poincar\'e inequality, we know that $\cE(x,x) - \lambda \Var_\pi[x] \geq 0$ with $\lambda > 0$, for all $x \in S^{\pi}$. Since $\cE(x,x) - \lambda \Var_\pi[x]$ is a quadratic form, this means that it is necessarily a sum-of-squares, i.e., $\cE(x,x) - \lambda \Var_\pi[x] \in \SOS_1(S^{\pi})$. By writing $\cE(x,x) = (\cE(x,x) - \lambda \Var_\pi[x]) + \lambda \Var_\pi[x]$ it suffices to prove that $\Var_\pi[x] = \|x-\1\|_{\pi}^2$ has the required decomposition \eqref{eq:EbAb}. This is proved in Lemma \ref{lem:sosdecompVarpi} in the Appendix.
%
\end{proof}

\begin{remark}
\label{rem:Pvanishgrad}
In the proof of Proposition \ref{feas} the only property we actually use about $P$ is that $P(\1) = 0$ and that $\nabla P(\1) \propto \pi$, which is a consequence of \eqref{eq::poly-approximates-eps}.
\end{remark}

None of our results so far give us any indication of how to choose $P(x)$, nor indeed whether it is possible for a polynomial to approximate $\Ent_\pi[x^2]$ in the manner that Theorem \ref{hierarchy-convergence} demands. This will be the topic of our next section.

\subsection{Choice of polynomial upper bound}\label{Poly}

The focus of this section is on obtaining polynomial upper bounds for the relative entropy $\Ent_\pi[x^2]=\sum_i\pi_ix_i^2\log(x_i^2)$ which are valid on $S^\pi$.
In this work we will only consider separable bounds of the form $P(x)=\sum_i\pi_ip_i(x_i)$, where we will require $p_i(t)\ge t^2\log(t^2)$ for all $t\in[0,1/\sqrt{\pi_i}]$.
We describe two ways of obtaining polynomial upper bounds on $q(t) = t^2\log (t^2)$ for $t > 0$. 
The first is elementary and consists simply of a truncated Taylor approximation. Even though this method yields reasonably good bounds on small chains, it cannot yield a \emph{convergent} sequence of lower bounds to $\alpha$.
The second is based on rational approximations of the logarithm. This more sophisticated method will allow us to obtain a hierarchy of sum-of-squares programs which provably converge to the logarithmic Sobolev constant of a given Markov process (Theorem \ref{thm:hierarchy-main}).

\subsubsection{Taylor series}\label{ref:taylor}
For $k\ge 2$, the Taylor expansion to odd order $2k-1$ of $q(t)=t^2\log(t^2)$ about $t=1$ gives an upper bound on $q(t)$ valid for every $t\ge0$.
To see this, note that by Taylor's theorem, the Lagrange form of the remainder term can be written 
\[q(t)-p_\textrm{Tay}^{2k-1}(t)=\frac{q^{(2k+2)}(\zeta_t)}{(2k+2)!}(t-1)^{2k+2},\] 
for some $\zeta_t$ betwen 0 and $t$.
We have $q^{(3)}(t)=4/t$, hence 
\[q^{(2k+2)}(t)=\frac{4(2k)!(-1)^{2k+1}}{t^{2k}}\]
 for $k \ge 2$.
It follows that \[q(t)- p_\textrm{Tay}^{2k-1}(t)=4(-1)^{2k-1}\frac{(t-1)^{2k+2}}{(2k+1)(2k+2)\zeta_t^{2k}},\] which is nonpositive.

Using $P(x) = \sum_i\pi_ip_\textrm{Tay}^{5}(x_i)$ in the SoS program \eqref{soscheese} will allow us to obtain the exact logarithm Sobolev constant for the simple random walk on the $n$-cycle for small $n$ (see Section \ref{sec:cycle}).
However, in view of the divergence of the truncated Taylor series for $\abs{t-1}>1$, the resulting sequence of sum-of-squares relaxations do not in general converge to the logarithmic Sobolev constant $\alpha$.

\subsubsection{Adaptive bounds via Pad\'e approximants}\label{Pade}

The $(m,k)$ \emph{Pad\'e approximant} of a smooth function $g(t)$ around $t=1$ is the rational function
\[r_{m,k}(t)=\frac{R(t)}{S(t)}\]
where $R$ and $S$ are polynomials of degree $m$ and $k$ respectively, such that $S(1)=1$ and such that  derivatives of $r_{m,k}(t)$ at $t=1$ agree with those of $g$ to as high an order as possible. For a generic smooth $g$, this means that derivatives of order ($0, 1\dots, m+k$) agree, i.e. $g(t)-r_{m,k}(t) = O\bigl((t-1)^{m+k+1}\bigr)$.
Despite being defined by the same category of data as the truncated Taylor series' (namely the first few derivatives of $g$ at $t=1$), in may situations Pad\'e approximants can be considered ``better'' approximations of $g$, in the sense that convergence can occur outside of the radius of convergence of the Taylor series, and moreover this convergence may be faster.
As the following result shows, this is indeed the case for the diagonal Pad\'e approximants of the logarithm function.

\begin{thm}\label{thm:main-pade-results}
    For each $m\ge0$ let $\overline{r_m}(t)$ be the $(m+1,m)$ Pad\'e approximant to $\log (t)$ at $t=1$.
    \begin{enumerate}[(i)]
        \item The denominator $S_m(t)$ of the rational function $\overline{r_m}(t)$ has no roots in $(0,\infty)$, hence $S_m$ is positive on this interval.
        \item For each $m\in\N$ there is a constant $B_m$ such that 
        \[0 \leq \overline{r_m}(t) - \log(t) \leq B_m (t-1)^2/t \qquad \forall t > 0, \]
        and such that $B_m \downarrow 0$ as $m\to \infty$. In fact, for $m\ge2$ we can take $B_m=\frac{1}{(m+1)^2}$.
        \item For $t>0$ and $m\ge0$, \[\overline{r_{m+1}}(t) \le \overline{r_m}(t).\]
    \end{enumerate}
\end{thm}
\begin{proof}
See \cite[Prop. 2.2]{squashedentanglement}. We remark that the cited proposition shows that $\overline{r_m}(t) - \log(t) \leq B_m \frac{(t-1)^2}{t} (\frac{t-1}{t+1})^{2m}$ so that convergence is exponentially fast on any compact interval of $(0,\infty)$.
\end{proof}

It remains to specify how to use the Pad\'e approximants $\overline{r_m}(t)$ in the sum-of-squares relaxation \eqref{soscheese}. The obvious difficulty is that $\overline{r_m}(t)$ is a rational function, and not a polynomial. To overcome this difficulty, we let the polynomial $P$ (i.e., its coefficients) be itself a decision variable in the optimization problem, subject to the constraint that $P(x)$ upper bounds $\Ent_{\pi}[x^2]$ on $S^{\pi}$. If we consider separable polynomials $P(x)=\sum_i\pi_ip_i(x_i)$, then this gives us the following optimization problem:

\begin{equation}\label{soscheese-vary}
    \min_{\gamma\in \R,\;p_i\in\R[t]_{2d}}\gamma \quad\st \quad \begin{cases} \gamma\mathcal{E}(x,x) - \sum_i\pi_ip_i(x_i) \in \SOS_d(S^{\pi}),\\
     p_i(t) \ge t^2\log(t^2) \quad \forall \,t\in[0,1/\sqrt{\pi_i}],\; \forall\,i\in\cX.\end{cases}
\end{equation}

The coefficients of $p_i$ are now variables of the optimization problems, which are constrained so that the inequality $p_i(t) \ge t^2\log(t^2)$ holds on the appropriate interval. Because we do not have a method to enforce this constraint directly, we replace $\log(t^2)$ by its rational Pad\'e approximant. The key observation is that given a rational function  $R(t)/S(t)$ which is an upper bound for $t^2\log(t^2)$, we can enforce the constraint $p_i(t) \ge R(t)/S(t)$ on the interval $[0,1/\sqrt{\pi_i}]$ by demanding a sum-of-squares certificate for the univariate polynomial  $S(t)p_i(t)-R(t)$. Such a certificate will naturally imply the stronger condition $p_i(t) \ge t^2\log(t^2)$. If we let
\[
\overline{r_m}(t) = \frac{(t-1) R_m(t)}{S_m(t)}
\]
be the $(m+1,m)$ Pad\'e approximant of $\log$ at $t=1$ (where $\deg R_m = \deg S_m = m$), then this yields the following optimization problem:

\begin{equation}\label{main-hierarchy}\tag{$\mathrm{H}_{d,m}$}
\boxed{
    \frac{1}{\alpha_{d,m}} := \min_{\gamma\in \R, \; p_i\in\R[t]_{2d}}\gamma \; : \; \begin{cases} \gamma\mathcal{E}(x,x) - \sum_{i=1}^n \pi_ip_i(x_i) \in \SOS_d(S^{\pi}),\\
        S_m(t) p_i(t) - 2t^2 (t-1) R_m(t) \in \SOS_{d+\lceil{m/2\rceil}}([0,\pi_i^{-1/2}])\\
        \qquad \qquad \qquad\qquad\qquad\qquad \qquad\qquad \;\;\; (i=1,\ldots,n).
        \end{cases}
}
\end{equation}
where $\SOS_{k}([a,b])$ is the set of polynomials of degree at most $2k$  that are nonnegative on the interval $[a,b]$. Since any such polynomial admits a sum-of-squares representation, this set can be described using semidefinite programming (see Section \ref{sec:univariate}).

We note that the semidefinite program \eqref{main-hierarchy} is indexed by two parameters: $d$, which is the degree of the sum-of-squares relaxation, and $m$ which is the degree of the rational approximation to the log function.

The next theorem shows that with large enough $m$ and $d$, one can find a polynomial $p$ of degree $2d$ such that $p_1=\dots=p_n=p$ is feasible for \eqref{main-hierarchy} and such that $\sum_{i} \pi_i p(x_i) \leq (1+\epsilon) \Ent_{\pi}[x^2]$ on $S^{\pi}$.

\begin{thm}\label{thm::poly-approx}
    Let $\pi_1,\dots,\pi_n >0$ be such that $\sum_{i=1}^n\pi_i=1$, and write $\pi_*=\min_i\pi_i$.
    For any $\epsilon \in (0,1)$, there exists $m\in\N$ and a univariate polynomial $p\in\R[t]$ such that the following is true:
    \begin{itemize}
    \item $p(t)\ge 2t^2\overline{r_m}(t)$ whenever $t\in[0,1/\sqrt{\pi_*}]$, and for all $x\in S^{\pi}$; and
    \item $\sum_{i=1}^n \pi_ip(x_i) \le (1+\epsilon)\Ent_\pi[x^2]$ for all $x \in S^{\pi}$.
    \end{itemize}
\end{thm}
\begin{proof}
    Write $N_{*}=\pi_*^{-1/2}$, and let $\epsilon_1>0$ be a constant depending on $\epsilon$ and on $N_{*}$ whose value we will determine later. 
    Choose $m\in\N$ large enough that $B_m<\epsilon_1$ (see Theorem \ref{thm:main-pade-results}).
    We obtain that for every $t\ge0$
    \begin{equation}\label{eq:qapprox}
        2t^2\,\overline{r_m}(t)-t^2\log(t^2) \le 2\epsilon_1t(t-1)^2.
    \end{equation}
    
    Applying the Weierstrass Approximation Theorem to the continuous function\footnote{Note that $\overline{r_m}(t) = (t-1) + O((t-1)^2)$ as $t\to 1$ since $\overline{r_m}(t)-\log(t)=O((t-1)^{2m+2})$.} $t\mapsto \frac{t\,\overline{r_m}(t) -(t-1)}{(t-1)^2} + \frac{\epsilon_1}{2}$, 
    we deduce the existence of a polynomial $q$ satisfying
    \begin{equation*}
    0 \le q(t) - \frac{t\,\overline{r_m}(t) -(t-1)}{(t-1)^2} \le \epsilon_1.
    \end{equation*}
    for all $t\in[0,N_{*}]$.
    Now, define $p(t) = 2t(t-1)^2q(t) + 2t(t-1)$, so that
    \begin{equation}
    \label{eq:ubp(t)}
    2t^2\,\overline{r_m}(t) \le p(t) \le 2t^2\,\overline{r_m}(t) + 2\epsilon_1t(t-1)^2
    \end{equation}
    holds for all $t\in[0,N_{*}]$. This polynomial $p$ certainly satisfies the first condition $p(t)\ge 2t^2\overline{r_m}(t)$. Combining \eqref{eq:qapprox} with the upper bound of \eqref{eq:ubp(t)} we have, for $x \in S^{\pi}$
    \begin{equation}\label{eq:qa}
        \sum_{i=1}^{n} \pi_i p(x_i) \le\Ent_{\pi}[x^2] + 4\epsilon_1\sum_{i=1}^n \pi_ix_i(x_i-1)^2.
    \end{equation}
    It remains to bound the last terms in \eqref{eq:qa} by $\epsilon\Ent_\pi[x]$, in order to obtain the desired inequality. For any $x\in S^\pi$, we have
    \begin{align*}\Ent_\pi[x^2] & = \sum\nolimits_i 2\pi_ix_i^2\log(x_i)\\ 
        &\ge \sum\nolimits_i 2\pi_ix_i^2\left(1-\frac{1}{x_i}\right)\\ 
        &= \sum\nolimits_i 2\pi_ix_i(x_i-1)\\ 
        &= \sum\nolimits_i \pi_i(x_i-1)^2 \ge \frac{1}{N_*}\sum\nolimits_i \pi_ix_i(x_i-1)^2,
    \end{align*}
    where we used $\sum_i\pi_ix_i^2=1$ in the penultimate step, and $x_i\le N_*$ in the final step.
    Choosing $\epsilon_1 = \frac{\epsilon}{4N_*}$ completes the proof.
\end{proof}

\subsection{Proof of Theorem \ref{thm:hierarchy-main-intro}}

We now have all the ingredients to prove the main result of the paper.

\begin{thm}\label{thm:hierarchy-main}
    Let $\mathcal{E},\,\pi$ be the Dirichlet form and stationary distribution of an irreducible Markov chain on $\cX=\{1,\dots,n\}$, with log-Sobolev constant $\alpha$.
    For $d\ge2$, $m\ge0$ the sum-of-squares program \ref{main-hierarchy} is feasible, and $0 < \alpha_{d,m}\le\alpha$.
    Moreover,
    \[\lim_{m\to\infty}\lim_{d\to\infty}\alpha_{d,m} = \alpha.\]
\end{thm}

\begin{proof}
    To prove feasibility, consider the choice $p_i(t)=p(t)=(t-1)(3t^2 - 2t+1)$ for each $i$. Then for every $t$, we have
    \[
    p(t)-2t^2\overline{r_m}(t)
    \ge
    p(t)-2t^2\overline{r_0}(t)
    =
    p(t)-2t^2(t-1) = 1,
    \]
    since $\overline{r_m}(t)$ is decreasing in $m$ by Theorem \ref{thm:main-pade-results} (ii). Thus these polynomials satisfy the second constraint of \eqref{main-hierarchy}. Furthermore, if we let $P(x) = \sum_{i=1}^n \pi_i p(x_i)$, then we see that $P(\1) = 0$ and $\nabla P(\1) \propto \pi$ so that from Proposition \ref{feas} (see also Remark \ref{rem:Pvanishgrad}) it follows  that $\alpha_{d,m} > 0$ for $d \geq 2$ and $m \geq 0$.

    To prove the second part, let $\epsilon\in(0,1)$ and choose $p\in\R[t]$ and $m\in\N$ as in Theorem \ref{thm::poly-approx}. Let $p_i = p$ for each $i\in\cX$.
    We can now apply Theorem \ref{hierarchy-convergence} with $P(x) = \sum_{i=1}^n \pi_i p(x_i)$, which guarantees the existence of $d_0\ge\frac{1}{2}(\deg p) +1$, such that $\alpha_{d,m} \ge (1-\epsilon)\alpha$ for every $d\ge d_0$.
\end{proof}
\begin{remark}[Size of the semidefinite program \eqref{main-hierarchy}]
\label{rem:size-sdp}
    The semidefinite program corresponding to $\delta_{d,m}$, is over the product of positive semidefinite cones
    \[
    \underbrace{\S_+^{\binom{n+d}{d}} \;\times\; \Bigl(\S_+^{\binom{n+d-1}{d-1}}\Bigr)^n}_{\SOS_d(S^{\pi})}
    \quad \times \quad
    \underbrace{\left( \S_+^{2d+m-3} \right)^n}_{\prod_{i} \SOS_{2d+m}([0,\pi_i^{-1/2}])}
     \]
    which has dimension $O(n^{2d}+m^2n))$  for fixed $d$.
\end{remark}
\begin{remark}[Monotonicity in $d$ and $m$]
    For fixed $d$, $\alpha_{d,m}$ is nondecreasing in $m$ (because the $(m+1,m)$ Pad\'e approximants are nonincreasing in $m$, see Theorem \ref{thm:main-pade-results}).
    Similarly for fixed $m$, $\alpha_{d,m}$ is nondecreasing in $d$, because any feasible $\big(\gamma, (p_i)_i\big)$ for $(\mathrm{H}_{d,m})$ is also valid for $(\mathrm{H}_{d',m})$ with $d'>d$.
    In particular, this means we can (in principle) consider only the diagonal steps of the hierarchy to get a hierarchy $(\alpha_{d,d})_d$ (indexed by a single parameter $d$) which converges monotonically to $\alpha$.
    In practice, since the size of the underlying semidefinite program for $\mathrm{H}_{d,m}$ scales exponentially in $d$, but only polynomially in $m$ (see Remark \ref{rem:size-sdp} above), it is desirable to keep the freedom to increase $m$ and $d$ independently.
\end{remark}

\begin{remark}[Strict feasibility of \eqref{main-hierarchy}]
\label{rem:strictfeas}
Even though we show in Theorem \ref{thm:hierarchy-main} that \eqref{main-hierarchy} is feasible as soon as $d \geq 2$ and $m \geq 0$, the semidefinite program is not, as written, \emph{strictly} feasible. Strict feasibility of a semidefinite program is a very useful property to have both in theory (e.g., for strong duality), and in practice, to avoid numerical issues. Thankfully, one can easily modify the semidefinite program \eqref{main-hierarchy} to make it strictly feasible. There are two reasons why \eqref{main-hierarchy} is not strictly feasible. The first reason is that the polynomial $\gamma \cE(x,x) - \sum_{i} \pi_i p_i(x_i)$ vanishes at $x=\1$ for any feasible $(\gamma, (p_i))$. To make the first constraint of the semidefinite program strictly feasible, it suffices to slightly change the definition of $\SOS_d(S^{\pi})$ in \eqref{eq:SOSOpi} and require that $\sigma_0 \in \Sigma(V_d)$ and $\sigma_1,\ldots,\sigma_n \in \Sigma(V_{d-1})$ where
\[
V_k = \{p \in \RR[x]_{k} : p(\1) = 0\},
\]
and $\Sigma(V_k) = \text{cone}\{ q^2 : q \in V_k \}$ (see \eqref{eq:SigmaV}).
The second reason that \eqref{main-hierarchy} is not strictly feasible, is that the univariate polynomials $S_m(t) p_i(t) - 2t^2(t-1)R_m(t)$ all vanish at $t=1$, so they cannot be in the interior of $\SOS_{d+\lceil m \rceil}(0,\pi_i^{-1/2})$. This can be easily fixed by factoring out a term $(t-1)$ in the definition of $\SOS_{2d+m}([0,\pi_i^{-1/2}])$. We omit the details here for brevity.
\end{remark}

\begin{remark}[A possible alternative approach]
    Instead of searching over \emph{polynomials} $p_i$ which are upper bounds for $2t^2\overline{r_m}(t)$, one could consider simply substituting $P(x)=2\sum_i\pi_ix_i^2\overline{r_m}(x_i)$ in \eqref{soscheese} and searching for a sum-of-squares nonnegativity certificate for 
    \[\gamma\cE(x,x)\prod_iS_m(x_i) - 2\sum_i\pi_ix_i^2(x_i-1)R_m(x_i)\prod_{j\neq i}S_m(x_j)\]
    on the hyperellipse $S^\pi$.
    The main drawback of this approach is that the above polynomial has degree $nm+3$, so we would need to search for certificates in $\QQ_{\frac{nm+3}{2}}(x)$. This means the size of the resulting semidefinite program grows like $\Omega(n^{mn+3})$. Even for very small values of $m,n$, this is impossible to implement in practice.
    In contrast, in our proposed method the semidefinite program has size $O(n^{2d}+m^2n)$, where $d$ is fixed in advance independent of $n$ and $m$. As we will see in section \ref{sec:numerics}, this approach works well in practice.
\end{remark}

\subsection{Exact rational lower bounds}\label{provable}

Computing the lower bounds $\alpha_{d,m}$ requires solving semidefinite programs numerically and at present, all efficient algorithms for solving semidefinite programs use floating point arithmetic. 
Hence the solutions they return are infeasible by a small margin. In this section we briefly describe how these numerical solutions can be turned into formal rational lower bounds on the logarithmic Sobolev constant. The techniques we use are similar to the rational rounding approach of \cite{Peyrl2008}.

The sum-of-squares program \eqref{main-hierarchy} can be put in standard semidefinite programming form as follows:
\begin{equation}\label{modnewsdpcheese}
    \begin{aligned}
        \text{minimize} \; \inp{c}{y} \; \quad \st\qquad& A\begin{pmatrix}y\\z\end{pmatrix}=b, \quad \smat(z)\succeq0\\
                                        &y\in \R^{N_1}, z\in\R^{N_2(N_2+1)/2}
    \end{aligned}
\end{equation}
where $A,b,c$ encodes the problem data, namely the transition kernel $K$ and the invariant distribution $\pi$. Here $\smat$ denotes the vectorization of the upper triangular part of a symmetric matrix. In most chains of interest the transition probabilities (and hence also the stationary probabilities) are rational numbers, so the corresponding data in $A,b,c$ are rational too. 

When a strictly feasible, bounded below SDP in the form \eqref{modnewsdpcheese} is provided as input to an interior-point solver, we can typically expect to obtain a floating point output $(\hat{y},\,\hat{z})$ satisfying 
\begin{equation}
\label{eq:approxfeasible_aff}
    \norm{A\begin{pmatrix}\hat{y}\\\hat{z}\end{pmatrix}-b}_\infty\leq\left(1+\norm{b}_\infty\right)\varepsilon_\mathrm{aff}
\end{equation}
for some constant $\varepsilon_\mathrm{aff}$ (typically $\approx 10^{-8}$). Because the linear equations are only satisfied approximately, we cannot expect that $\inp{c}{\hat{y}}$ is a true upper bound on $\alpha(K)\inv$.

In this case, following \cite{Peyrl2008} we seek to round $(\hat{y},\,\hat{z})$ to rational numbers satisfying the affine constraints exactly.
This can be done in two steps:

\begin{itemize}
\item[1.] First, we convert the floating-point solution $(\hat{y},\,\hat{z})$ to rational numbers. Every floating-point number is already rational, with denominator some power of 2, so this can be done exactly by simply changing the internal representation our computer uses to represent these numbers.
\item[2.] Next, we project $(\hat{y},\,\hat{z})$ onto the affine subspace defined by the constraints of \eqref{modnewsdpcheese}:
\[\mathcal{A}=\left\{(y,\,z) \in \mathbb{Q}^{N_1} \times \mathbb{Q}^{N_2(N_2+1)/2} \mid A\begin{pmatrix}y\\z\end{pmatrix}=b\right\}.\]
The orthogonal projection map onto this space is given by 
\[\Pi_{\mathcal{A}}\begin{pmatrix}y\\z\end{pmatrix} = \begin{pmatrix}y\\z\end{pmatrix} - A^\top\left(AA^\top\right)\inv\left(A\begin{pmatrix}y\\z\end{pmatrix} - b\right).\]
Note that this projection operator uses only rational arithmetic, and can be implemented exactly in a computer algebra system.
Therefore when applied to $(\hat{y}, \hat{z})$, this projection produces a rational point 
\begin{equation}
\label{eq:ystarzstar}
    \begin{pmatrix}y^*\\z^*\end{pmatrix}=\Pi_\mathcal{A}\begin{pmatrix}\hat{y}\\\hat{z}\end{pmatrix}
\end{equation}
satisfying the affine constraints exactly.
\end{itemize}
However the new point $(y^*,z^*)$ may not be feasible for the semidefinite program \eqref{modnewsdpcheese}, in particular $\texttt{VecToSymMat}(z^*)$ may not be positive semidefinite.
To remedy this situation, consider the slightly stronger semidefinite program:
\begin{equation}\label{modnewsdpcheese-eps}
    \begin{aligned}
        \text{minimize }\inp{c}{y} \quad \st\qquad& A\begin{pmatrix}y\\z\end{pmatrix}=b, \quad \smat(z)\succeq \varepsilon I\\
                                        &y\in \R^{N_1}, z\in\R^{N_2(N_2+1)/2},
    \end{aligned}
\end{equation}
where $\varepsilon > 0$ is a parameter. The next proposition shows that by rounding a floating-point solution of \eqref{modnewsdpcheese-eps} with an appropriately chosen parameter $\varepsilon > 0$, one can obtain a rational feasible point to \eqref{modnewsdpcheese} which satisfies the constraints exactly.
\begin{prop}
\label{prop:rounding}
Consider the semidefinite program \eqref{modnewsdpcheese-eps} where $A \in \R^{r\times (N_1 + N_2(N_2+1)/2)}$ is assumed to have full row rank $r$.
Assume $(\hat{y},\hat{z})$ is an approximate rational feasible point of \eqref{modnewsdpcheese-eps} with 
\[
\varepsilon = \sqrt{\frac{2r}{\lambdamin(AA^{\top})}} (1+\|b\|_{\infty}) \varepsilon_{\text{aff}}
\]
i.e.,
\begin{equation}
\label{eq:approxfeasible}
\smat(\hat{z}) \succeq \varepsilon I \text{ and } \|A \begin{sm} \hat{y}\\ \hat{z}\end{sm} - b\|_{\infty} \leq (1+\|b\|_{\infty})\varepsilon_{\text{aff}}.
\end{equation}
Then $(y^*,z^*)$ defined by \eqref{eq:ystarzstar} is a rational feasible point to \eqref{modnewsdpcheese}.
\end{prop}
\begin{proof}
See Appendix \ref{sec:proofrounding}.
\end{proof}

\section{Beyond logarithmic Sobolev inequalities}
\label{sec:mlsi}

In this section we show how the techniques developed in the previous section can be applied for other entropic functional inequalities. We show how to use sum-of-squares relaxations combined with Pad\'e rational approximants to obtain bounds on \emph{modified logarithmic Sobolev constants} (also known as the entropy constant), and \emph{strong data processing inequalities}.

\subsection{Algorithmic bounds on the entropy constant}

Given an irreducible Markov kernel $K$ with stationary distribution $\pi$, as an alternative to \eqref{eq:lsineq}, one can try to prove \emph{modified} logarithmic Sobolev inequalities (MLSI) \cite{Ledoux2001} of the form
\begin{equation}\label{eq:mlsi-orig}
    \cE(x, \log x) \ge \rho \Ent_\pi[x]\qquad  \forall x\in \R^n_+.
\end{equation}
The largest constant $\rho$ for which \eqref{eq:mlsi-orig} holds is called the \emph{entropy constant} of the Markov semigroup $P_t$ associated to $K$ via $P_t=e^{t(K-I)}$.
Modified log-Sobolev inequalities are important because it can be shown \cite[Theorem 2.4]{Tetali2006} that the entropy constant has an equivalent characterization as the best constant $\rho$ such that 
\begin{equation}\label{eq:mlsi-relent-contracts}
    D(\mu_t \Vert \pi) \le e^{-\rho t}D(\mu_0 \Vert \pi)
\end{equation}
holds for all initial distributions $\mu_0$ on $\cX$, where $\mu_t:=P_t\mu_0$. Therefore, the entropy constant of a Markov semigroup bounds the \emph{relative entropy} mixing time of the corresponding Markov process.

For every irreducible kernel $K$, one has $2\alpha(K)\le \rho(K) \le 2\lambda(K)$, and if $K$ is reversible, the first inequality can be strengthened to $4\alpha(K)\le \rho$ \cite[Proposition 1.10 \& Remark 1.11]{Montenegro2006}. Thus the log-Sobolev constant always provides a lower bound on the entropy constant. However, one can have $\alpha \ll \rho$;  for example in the case of the random walk on the group of permutations $S_n$ of $[n]$ which at each jump moves by a uniformly sampled transposition, see \cite[Example 3.12]{Tetali2006}.

Note that for any $x\in\R^n_{++}$ and $\gamma>0$, we have 
\[\cE(\gamma x, \log \gamma x) = \gamma \cE(x, \log x + \log(\gamma)\1) = \gamma \cE(x, \log x) + \log(\gamma)\cE(x, \1)=\gamma \cE(x, \log x),\]
so by homogeneity we may restrict the variable $x$ appearing in \eqref{eq:mlsi-orig} to the simplex
\[\Delta^\pi := \{x\in\R^n_+ : \E_\pi[x]=1\}.\]
We can now rearrange \eqref{eq:mlsi-orig}, using the definitions of $\cE$ and $\Ent_\pi$, to obtain the equivalent inequality
\begin{equation}\label{eq:mlsi-reform}
    \sum_{i,j\in\cX} \pi_i K_{ij} x_i [\bar\rho \log x_i - \log x_j] \ge0\qquad \forall x\in\Delta^\pi,
\end{equation}
where $\bar\rho:=1-\rho$.

Suppose that for each pair $(i,j)\in\cX\times\cX$, $p_{ij}(t, s)$ is a \emph{bivariate} polynomial satisfying
\begin{equation}\label{eq-bivar-mlsi}
    p_{ij}(t, s) \le \bar\rho \,t\log t - t\log s\qquad\forall(t,s) \in [0, \pi_i\inv]\times[0,\pi_j\inv],
\end{equation}
for some fixed $\bar\rho > 0$.
Then the polynomial inequality 
\begin{equation}\label{eq:poly-mlsi}
    \sum_{i,j\in\cX} \pi_i K_{ij} p_{ij}(x_i, x_j) \ge 0 \qquad \forall x\in\Delta^\pi
\end{equation}
implies \eqref{eq:mlsi-reform}.
In order to construct sum-of-squares relaxations of this type of inequality, we introduce the set
\begin{equation}\label{eq:SOSDelta}
\begin{aligned}
\SOS_d(\Delta^{\pi}) := \Biggl\{ f \in \RR[x]_{2d} : & \; f(x) = \sigma_0(x) + \sum_{i=1}^n x_i \sigma_i(x) + \phi(x) (\E_\pi[x]-1)\\
& \qquad\qquad \text{ where } \sigma_0 \in \Sigma[x]_{2d}, \; \sigma_i \in \Sigma[x]_{2d-2} \;\;(i=1,\ldots,n)\\
& \qquad \qquad \qquad \qquad \phi \in \RR[x]_{2d-1} \Biggr\}.
\end{aligned}
\end{equation}
$\SOS_d(\Delta^\pi)$ is the degree-$d$ sum-of-squares relaxation of the set of polynomials which are nonnegative on $\Delta^\pi$.
The following sum-of-squares programs, indexed by $d\ge \frac{1}{2} \max_{i,j} \{\deg p_{ij}\}$, are a sequence of increasingly general sufficient conditions for \eqref{eq:poly-mlsi}:
\begin{equation}\label{eq:sos-mlsi}
    \sum_{i,j\in\cX} \pi_i K_{ij} p_{ij}(x_i, x_j) \in \SOS_d(\Delta^\pi).
\end{equation}

Now we consider how to obtain sum-of-squares certificates of the bivariate inequality \eqref{eq-bivar-mlsi}. In Section \ref{Pade}, we saw that the rational functions $\overline{r_m}(s)=\frac{(s-1)R_m(s)}{S_m(s)}$ [the $(m+1,m)$ Pad\'e approximants to $\log(s)$ around $s=1$] are a convergent sequence of upper bounds for $\log(s)$.
In order to prove inequalities of the form \eqref{eq-bivar-mlsi}, we will also need rational \emph{lower bounds} on $\log(t)$. Observe that
$\log(t) = -\log(1/t) \ge -\overline{r_m}(1/t)$.
Therefore \[t\log(t) \ge -t\,\overline{r_m}(1/t)= \frac{(t-1)\overline{R_m}(t)}{\overline{S_m}(t)},\] where $\overline{S_m}(t):=t^mS_m(1/t)$ and $\overline{R_m}(t):=t^mR_m(1/t)$ are degree-$m$ polynomials.

A sufficient condition for \eqref{eq-bivar-mlsi} is therefore 
\[p_{ij}(t, s) \le - \bar\rho \,t\,\overline{r_m}(1/t) - t\,\overline{r_m}(s) = \bar\rho\frac{(t-1)\overline{R_m}(t)}{\overline{S_m}(t)} - \frac{t(s-1)R_m(s)}{S_m(s)} \qquad\forall(t,s) \in [0, \pi_i\inv]\times[0,\pi_j\inv].\]
Clearing denominators (recalling that the polynomial $S_m(t)$ is positive for $t\ge 0$), this is equivalent to the polynomial inequality
\[\bar\rho(t-1)\overline{R_m}(t)S_m(s) - t(s-1)R_m(s)\overline{S_m}(t) - p_{ij}(t, s)\overline{S_m}(t)S_m(s) \ge 0 \qquad\forall(t,s) \in [0, \pi_i\inv]\times[0,\pi_j\inv].\]
This polynomial inequality can be enforced as a sum-of-squares constraint using the set 
\begin{equation*}
    \SOS_k\Bigl([0, \pi_i\inv]\times[0,\pi_j\inv]\Bigr) := \Bigl\{  \sigma_0 + t\,(1-\pi_it)\sigma_1 + s(1-\pi_j s)\sigma_2 \;\mid\; \sigma_0 \in \Sigma[t,s]_{2k}, \; \sigma_1,\sigma_2 \in \Sigma[t,s]_{2k-2} \Bigr\}
    \end{equation*}
of degree-$k$ sum-of-squares relaxations of the set of polynomials which are nonnegative on $[0, \pi_i\inv]\times[0,\pi_j\inv]$.

We arrive at the following sum-of-squares optimization problem, parametrized by positive integers $d, m$:
\begin{equation*}
    \boxed{
        \rho_{d,m} := 1 - \min_{\bar\rho\in \R, \; p_{ij}\in\R[t,s]_{2d}}\bar\rho \; : \; \begin{cases} \sum_{i,j\in\cX} \pi_i K_{ij} p_{ij}(x_i, x_j) \in \SOS_d(\Delta^\pi),
            \vspace{2mm}\\
            \bar\rho(t-1)\overline{R_m}(t)S_m(s) - t(s-1)R_m(s)\overline{S_m}(t) - p_{ij}(t, s)\overline{S_m}(t)S_m(s) \\ \qquad\qquad\qquad {} \in \SOS_{d+m}\Bigl([0, \pi_i\inv]\times[0,\pi_j\inv]\Bigr)
             \qquad (i,j=1,\ldots,n).
            \end{cases}
    }
\end{equation*}
For every $m,d$, the numbers $\rho_{d,m}$ are upper bounds for the true entropy constant $\rho(K)$, and $\rho_{d,m}$ can be computed as the solution to a semidefinite program of size $O(n^{2d}+m^4n^2)$ for fixed $d$.

\subsection{Strong data processing inequalities}\label{sec:sdpi}
\newcommand{\cY}{\mathcal{Y}}
A fundamental inequality in information theory is the \emph{data processing inequality}, which states that if $W:\cX\to\cY$ is a channel, then for any two input distributions $\mu,\pi$ on $\cX$ we have
\[
D(\mu W\| \pi W) \leq D(\mu\|\pi).
\]
For a specific channel $W$ and reference input distribution $\pi$, such an inequality can usually be strengthened. We say that the pair $(\pi,W)$ satisfies a \emph{strong data processing inequality (SDPI)} if there is a constant $\delta < 1$ such that
\begin{equation}\label{eq:SDPI1}
D(\mu W\| \pi W) \leq \delta D(\mu\|\pi)
\end{equation}
for all probability distributions $\mu$ on $\cX$ \cite{Ahlswede76,Polyanskiy17}. 
The smallest such $\delta=\delta^*(\pi,W)$ is the SDPI constant of $(\pi,W)$.
For example, the binary symmetric channel with noise $\epsilon$ and a uniform source is known to have SDPI constant $\delta^*(\mathrm{Bern}(\frac{1}{2}), \mathrm{BSC}_\epsilon)=(1-2\epsilon)^2$ \cite{Ahlswede76}. 
Strong data processing inequalities, and more generally the contraction properties of discrete channels, have received a lot of attention recently in the information theory community \cite{Raginsky13,courtade2013outer,anantharam2014hypercontractivity,Raginsky15,Polyanskiy17,polyanskiy2020application,Polyanskiy21}, and have been applied e.g., to obtain various converse results.

The similarity of \eqref{eq:SDPI1} and \eqref{eq:mlsi-relent-contracts} suggests a connection between LSI/MLSI and SDPI. Indeed, such a connection was made precise by Raginsky in \cite{Raginsky13}. First, one defines the cascade channel $W^\sharp W : \cX\to \cX$ by 
$W^\sharp W(j | i) = \sum_{k\in\cY}\frac{W(k| i)W(k | j) \pi_j}{(\pi W)_k}$ for $i,j\in\cX$.
$W^\sharp W$ satisfies the detailed balance equations with respect to $\pi$, so it is a $\pi$-reversible Markov kernel on $\cX$. Then one has $\alpha(W^\sharp W) \le 1-\delta^*(\pi, W) \le \rho(W^\sharp W)$ \cite[Theorems 3 \& 4]{Raginsky13}.

In the paper \cite{sdpi-isit}, the authors present a hierarchy of semidefinite programming relaxations  which give certified upper bounds on the strong data processing (SDPI) constant of a discrete channel. Moreover it is shown that the hierarchy converges to the true SDPI constant. The relaxations developed in \cite{sdpi-isit} should be seen as the SDPI analogues of the relaxations developed in the present work for log-Sobolev and modified log-Sobolev inequalities. 

\section{Numerical results}\label{sec:numerics}

In this section we numerically illustrate the ideas developed in the previous sections.
In cases where the exact log-Sobolev constant $\alpha$ is not known (which are all cases in this section apart from subsection \ref{complete}) we use a method developed in section \ref{upper_bounds} to find an upper bound $\overline \alpha$ for $\alpha$ in order to quantify approximation accuracy.

In each experiment (sections \ref{complete} through \ref{sec::3point}) we set $2d=6$ (the degree of the polynomials $p_i$ bounding the logarithmic term) and use the Pad\'e relaxation approach of section \ref{Pade} with $m=5$ (the degree of the Pad\'e upper bound).

\subsection{Obtaining upper bounds for comparison}\label{upper_bounds}
The exact log-Sobolev constant is unknown for most finite state Markov chains.
In order to judge the accuracy of our sum-of-squares method, we will need a way of obtaining good upper bounds on the log-Sobolev constant of a given Markov chain.
One way to do this is to search for $x\in\R^n$ for which
\[ \Psi(x) := \frac{\mathcal{E}(x,x)}{\Ent_\pi[x^2]} = \frac{\sum_{i,j\in\cX}\pi_iK_{ij}(x_i-x_j)^2}{\sum_{i\in\cX} \pi_i x_i^2\log\bigl(\frac{x_i^2}{\norm{x}_\pi^2}\bigr)}\]
is small.
Indeed the log-Sobolev constant of the chain under consideration is exactly $\alpha = \inf_{x\in\R^n} \Psi(x)$.

This minimization problem is nonconvex and difficult in general. We use Newton's method to perform a local optimization, initialized at a randomly chosen initial point $x^*$, to obtain a local minimizer $x^*_\mathrm{min}$.
This is repeated $R$ times, with different initializations, and the algorithm returns the best $x^*_\mathrm{min}$ found.
We then take $\overline{\alpha} = \Psi(x_\mathrm{min})$, a valid upper bound on the real log-Sobolev constant $\alpha$.
Experimentally, we find that $R=100$ is usually sufficient for random Markov chains on $n=10$ vertices (see the numerical results in section \ref{sec:random}).

\subsection{The complete graph}\label{complete}

The log-Sobolev constant for the simple random walk on the complete graph with $n$ vertices is known to be $\alpha=\frac{n-2}{(n-1)\log (n-1)}$ (see \cite[Corollary A.4]{P.Diaconis1996}).
In Table \ref{tbl:compl} lower bounds obtained from a Pad\'e relaxation with $d=3$ and $m=5$ are listed, along with the relative error of this approximation, defined in this case as $\epsilon_\mathrm{rel} = \alpha/\alpha_{3,5} - 1$.
For comparison, we have also tabulated the relative errors of a truncated Taylor series based approximation, as described in section \ref{ref:taylor}.

While performing the computations, no attempt was made to take advantage of the large amount of symmetry in the structure of these Markov chains. Doing this would have sped up computations significantly, but requires some effort to implement.
In section \ref{sec:cycle}, we will see an example where symmetry is exploited in order to compute the log-Sobolev constant for chains on larger state spaces than we consider here.

\begin{table}[ht]
    \centering
	\begin{tabular}[b]{|c|l|l||l|}
            \hline
            $n$   & $\alpha_{3,5}$  & $\epsilon_\mathrm{rel}$ (for $\alpha_{3,5}$) & $\epsilon_\mathrm{rel}$ (Taylor)\\
            \hline
            3   & 0.72134751987  & $7.96 \times 10^{-10}$ & $2.76\times 10^{-4}$\\
            4   & 0.6068261485   & $4.25 \times 10^{-9}$ & $2.09\times 10^{-3}$\\
            5   & 0.541010629    & $2.16 \times 10^{-8}$ & $6.21\times 10^{-3}$\\
            6   & 0.497067908    & $7.95 \times 10^{-8}$ & $1.28\times 10^{-2}$\\
            7   & 0.46509209     & $2.22 \times 10^{-7}$ & $2.18\times 10^{-2}$\\
            8   & 0.44048407     & $5.06 \times 10^{-7}$ & $3.31\times 10^{-2}$\\
            9   & 0.4207856      & $1.02 \times 10^{-6}$ & $4.62\times 10^{-2}$\\
            10  & 0.4045500      & $1.85 \times 10^{-6}$ & $6.12\times 10^{-2}$\\
            11  & 0.3908638      & $3.13 \times 10^{-6}$ & $7.76\times 10^{-2}$\\
            12  & 0.3791184      & $5.06 \times 10^{-6}$ & $9.51\times 10^{-2}$\\
            13	& 0.3688909	     & $7.81 \times 10^{-6}$ & $1.14\times 10^{-1}$\\
            \hline
        \end{tabular}
    \caption{Computed lower bounds on the log-Sobolev constant of $K_n$. The rightmost column gives the relative error coming using a degree-5 Taylor series approximation.}
    \label{tbl:compl}
\end{table}


\subsection{The Petersen graph}

\begin{figure}[ht]
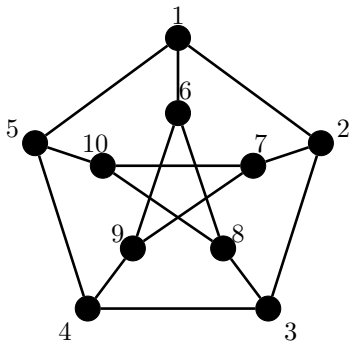

    \ctikzfig{png/petersen}
    \centering
    \caption{The Petersen graph}
    \label{fig:petersen}
\end{figure}

The Petersen graph is a famous graph with 10 vertices and 15 edges. 
It exhibits a lot of symmetry -- its automorphism group is isomorphic to $S_5$.
We consider the simple random walk on this graph.
Its spectral gap is $\lambda=2/3$.
Using a Pad\'e relaxation with $d=3$ and $m=5$, we obtain
\[\alpha_{3,5} = 0.306638... \approx 0 . 4 5 9 9 5 7\cdot\lambda.\]
Using the technique discussed in Section \ref{upper_bounds}, we obtain the upper bound
\[\overline \alpha = 0.306652... \approx 0 . 4 5 9 9 7 9\cdot\lambda.\]
We conclude that for the Petersen graph $\alpha\in (0.306638, 0.306652)$.

\subsection{Random graphs}\label{sec:random}

We generate 6 random graphs on 10 vertices with 15, 20, 25, 30, 35, and 40 edges respectively, as follows: the first graph is chosen uniformly at random among all graphs with 10 vertices and 15 edges. For each new graph, 5 new edges are selected uniformly at random.

\begin{figure}[htbp]
    \centering
    \begin{subfigure}[b]{0.25\textwidth}
        \caption{15 edges}
        \includegraphics[height=3cm]{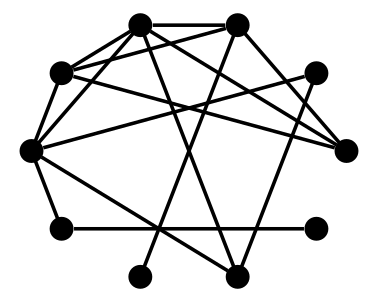}
        \label{rand:15}
    \end{subfigure}
    \hspace{0.05\textwidth}
    \begin{subfigure}[b]{0.25\textwidth}
        \caption{20 edges}
        \includegraphics[height=3cm]{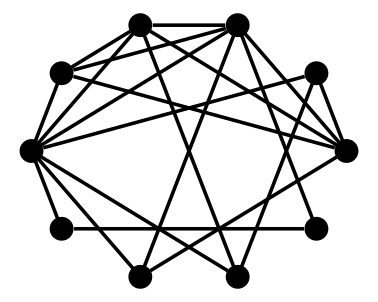}
        \label{rand:20}
    \end{subfigure}
    \hspace{0.05\textwidth}
    \begin{subfigure}[b]{0.25\textwidth}
        \caption{25 edges}
        \includegraphics[height=3cm]{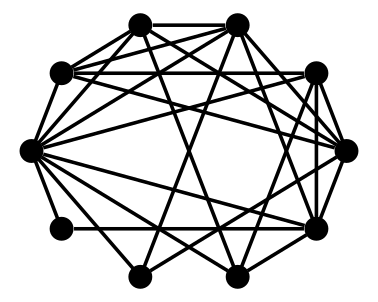}
        \label{rand:25}
    \end{subfigure}
    
    \begin{subfigure}[b]{0.25\textwidth}
        \caption{30 edges}
        \includegraphics[height=3cm]{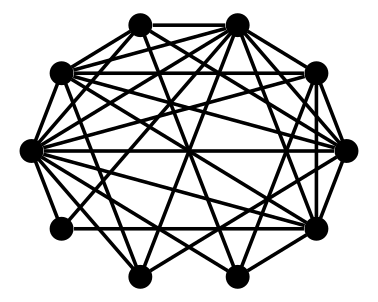}
        \label{rand:30}
    \end{subfigure}
    \hspace{0.05\textwidth}
    \begin{subfigure}[b]{0.25\textwidth}
        \caption{35 edges}
        \includegraphics[height=3cm]{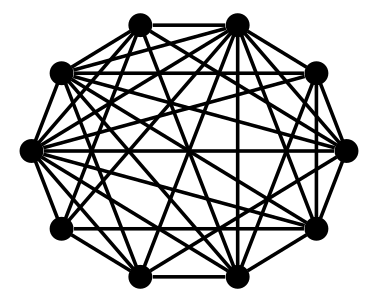}
        \label{rand:35}
    \end{subfigure}
    \hspace{0.05\textwidth}
    \begin{subfigure}[b]{0.25\textwidth}
        \caption{40 edges}
        \includegraphics[height=3cm]{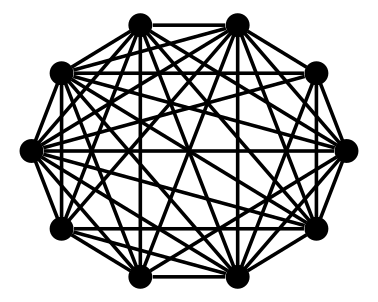}
        \label{rand:40}
    \end{subfigure}
    \caption{Random graphs with 10 vertices}\label{fig:rand}
\end{figure}

The corresponding sum-of-squares lower bounds $\alpha_{3,5}$ are listed in Table \ref{tab:rand}.
Upper bounds $\overline \alpha$ are obtained using the method in Section \ref{upper_bounds}, with $R=100$.

\begin{table}[htbp]
    \begin{center}
        \caption{Upper and lower bounds on log-Sobolev constants for simple random walk on the graphs in Figure \ref{fig:rand}}
        \label{tab:rand}
        \begin{tabular}{|c|c|c|c|c|c|c|}
            \hline
            Graph (c.f. Fig \ref{fig:rand})  &   Edges  &   $\lambda/2$  &  $\alpha_{3,5}$  &  $\overline \alpha$   & $\epsilon_\mathrm{rel}  $\\
            \hline
            (A)  &   15  &      0.11415158  &  0.09708470  &  0.09708648  &  $1.83 \times 10^{-5}$\\
            (B)  &   20  &      0.23223992  &  0.18379917  &  0.18380145  &  $1.24 \times 10^{-5}$\\
            (C)  &   25  &      0.32038757  &  0.27180172  &  0.27182406  &  $8.22 \times 10^{-5}$\\
            (D)  &   30  &      0.37911247  &  0.30284158  &  0.30285960  &  $5.95 \times 10^{-5}$\\
            (E)  &   35  &      0.41522081  &  0.34270699  &  0.34271490  &  $2.31 \times 10^{-5}$\\
            (F)  &   40  &      0.43176480  &  0.35569664  &  0.35570480  &  $2.29 \times 10^{-5}$\\
            \hline
        \end{tabular}
    \end{center}
\end{table}

\subsection{The three-point stick}\label{sec::3point}

\begin{figure}[htbp]
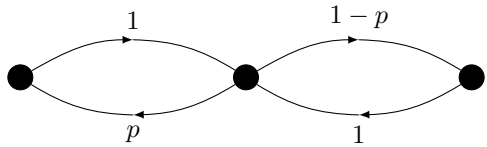

    \centering
    \ctikzfig{png/3stick}
    \caption{A Markov chain on the 3-point stick.
    }\label{3stick}
\end{figure}

Consider the family of three-point Markov chain in Figure \ref{3stick}, parametrized by $p\in(0,1/2)$.
It satisfies $\lambda=1$ independently of $p$.
Similar families of chains on the three-point space were studied in \cite{Chen2008} (although not this one).
For $p=1/2$, our chain coincides with one studied in \cite[Theorem 4.1]{Chen2008}, where it was shown that $\alpha=1/2$.

We obtain lower bounds $\alpha_{3,5}$ and upper bounds $\overline \alpha$ on the log-Sobolev constant for a range of values of $p$ in $[0.001,\, 0.49]$.
Once again, we take $d=3$ and $m=5$ when computing the sum-of-squares relaxation $\alpha_{3,5}$.
These are plotted in Figure \ref{3stickresults}.
The relative error (defined as $\epsilon_\mathrm{rel}=\frac{\overline \alpha - \alpha_{3,5}}{\alpha_{3,5}}$) ranges from $\epsilon_\mathrm{rel}=1.57\times10^{-8}$ for $p=0.49$ to $\epsilon_\mathrm{rel}=7\times10^{-2}$ for $p=0.001$.

\begin{figure}[htbp]
    \centering
    \includegraphics[width=0.7\textwidth]{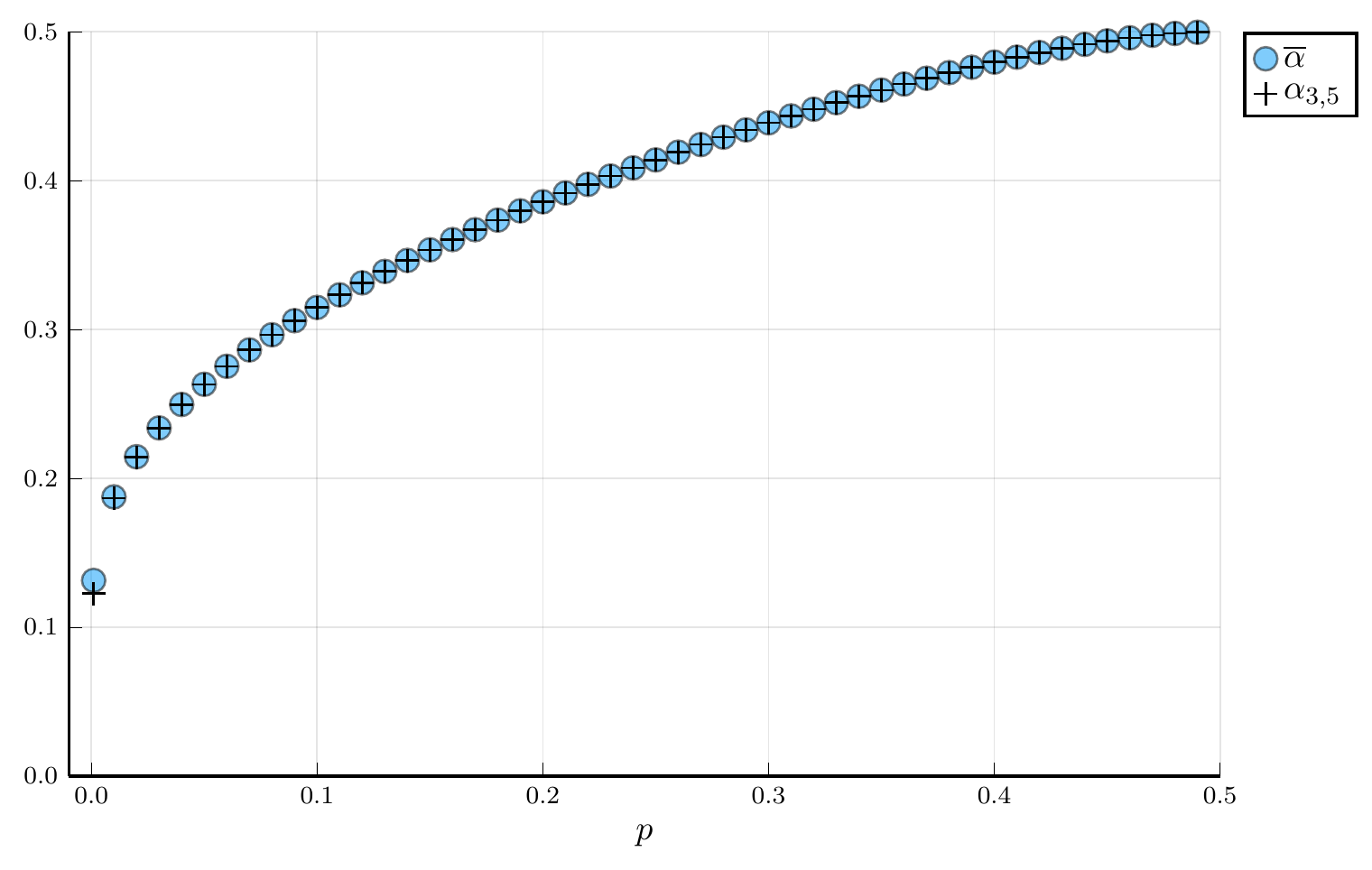}
    \caption{
        Upper and lower bounds on the log-Sobolev constant of the Markov chain shown in Figure \ref{3stick}, parametrized by $p$.
        The leftmost value of $p$ is 0.001, and the other values shown in this plot are $p=i/100,\; i=1, \dots, 49$. 
    }
        \label{3stickresults}
\end{figure}

\section{Exact log-Sobolev constants for the odd $n$-cycle}\label{sec:cycle}

Consider the simple random walk on the $n$-cycle, $n\ge 3$ defined by the transition kernel
\[
K_{i, i\pm 1} = 1/2 \;\; \forall i \in \ZZ_n
\]
and $K_{ij} = 0$ otherwise. This chain has Dirichlet form 
\[\mathcal{E}(x,x)=\frac{1}{2n}\sum_{i\in\Z_n}(x_i - x_{i+1})^2,\]
its stationary distribution is the uniform distribution $\pi=\1/n$, and it has spectral gap $\lambda=1-\cos\bigl(\frac{2\pi}{n}\bigr)$.
The log-Sobolev constant is known for $n$ even \cite{Chen2003}, and when $n=5$ \cite{Chen2008}, in which case it is equal to half the spectral gap. Its value was not yet known for odd $n$ larger than 5.
We will show how the technique developed in Section \ref{sec:approach} of this paper for obtaining lower bounds on logarithmic Sobolev constants can use to prove the following:
\begin{thm}
The logarithmic Sobolev constant for the simple random walk on the $n$-cycle is equal to $\frac{\lambda}{2}=\frac{1}{2}(1-\cos\frac{2\pi}{n})$ for $n\in\{5,7,9,\dots,21\}$.
\end{thm}
The proofs are too long to be checked manually, but are made available online, along with the small amount of code required to carry out a verification with the open-source computer algebra system SageMath \cite{sagemath}.

Showing that the log-Sobolev constant of the $n$-cycle is equal to $\lambda/2$ amounts to proving the inequality
\[
\frac{2}{\lambda} \cE(x,x) - \frac{2}{n} \sum_{i=1}^n x_i^2 \log(x_i) \geq 0
\]
for all $x \in \RR^{n}$ satisfying $\norm{x}_\pi^2=\frac{1}{n}\norm{x}^2= 1$. Following the approach described in Section \ref{sec:approach}, we will prove instead a stronger inequality where the logarithmic term $2x_i^2 \log(x_i)$ is upper bounded by its Taylor expansion to order 5, namely
\[
p_\mathrm{Tay}^5(t) = 2(t-1)+3(t-1)^2+\frac{2}{3}(t-1)^3-\frac{1}{6}(t-1)^4+\frac{1}{15}(t-1)^5.
\]
If we define the polynomial
\[
F(x) = \frac{1}{1-\cos(2\pi/n)}\sum_{i\in\Z_n}(x_i - x_{i+1})^2 - \sum_{i=1}^np_\mathrm{Tay}^5(x_i),
\]
our goal is to show that $F(x)$ is a sum-of-squares on the sphere $\{x \in \RR^n : \norm{x}^2 = n\}$. In other words, we want to solve the sum-of-squares feasibility problem:
\begin{equation}
\label{eq:cyclesosfeas}
\exists h \in \RR[x]_{n,4} \; \text{ such that } \; F(x) + \left(\norm{x}^2-n\right)h(x) \in \Sigma\left(W\right)
\end{equation}
where $W \subset \R[x]_{3}$. This can be phrased as a semidefinite feasibility problem and solved numerically much like the examples in Section \ref{sec:numerics}.

\subsection{Facial reduction}

There are two barriers to extracting a proof from a floating-point SDP solution using the rational rounding method described in Subsection \ref{provable}.

Firstly, the data defining the problem are not in general rational, since they include the number $\cos(2\pi/n)$.
However this can be easily remedied by working in the number field $\Q[\cos(2\pi/n)]$ instead.
Secondly, in order to produce formal certificates, we want the sum-of-squares feasibility problem \eqref{eq:cyclesosfeas} to be \emph{strictly} feasible (see Section \ref{provable}). 
Note that Remark \ref{rem:strictfeas} about strict feasibility  does not apply here because the variable $\gamma$ from \eqref{main-hierarchy} is now fixed to be $2/\lambda$.
The next theorem gives necessary conditions on the subspace $W$ so that the feasibility problem \eqref{eq:cyclesosfeas} is strictly feasible.

\begin{restatable}{thm}{cyclesosWmax}\label{thm:cyclesosWmax}
If \eqref{eq:cyclesosfeas} is strictly feasible, then necessarily 
\begin{equation}\label{eq:maximal_subspace}
    W \;\subseteq\; W_\mathrm{max} :=\{q\in\R[x]_{3} \mid q(\1) = 0,\; \phi^\top\nabla q(\1) = 0,\; \psi^\top\nabla q(\1) = 0\}.
\end{equation}
where $\phi,\,\psi\in\R^n$ are the vectors $\phi_i = \cos(2\pi i/n)$, $\psi_i = \sin(2\pi i/n)$.
\end{restatable}
\begin{proof}
For $h \in \RR[x]_{n,4}$ let 
\begin{equation}
\label{eq:Fhcycle}
F_h(x) = F(x) + \left(\norm{x}^2-n\right)h(x).
\end{equation}
The theorem follows from the fact that for any $h$ feasible for \eqref{eq:cyclesosfeas}, $F_h(\1) = 0$ and $\phi^{\top} \nabla^2 F_h(\1) \phi = 0$ and $\psi^{\top} \nabla^2 F_h(\1) \psi = 0$. The details are deferred to Appendix \ref{sec:appcycle}.
\end{proof}


In practice, we used the following choice of subspace $W$, which is a strict subspace of $W_{\max}$:
\begin{equation}
\label{eq:cyclesubspaceW}
    W = \left\{q\in\R[x]_{3}\; \middle\vert \; \begin{array}{l}
    q(\1) = 0,\\
     \phi^\top\nabla q(\1) = 0,\\
     \psi^\top\nabla q(\1) = 0,\\
    \dfrac{d^3q}{dx_idx_jdx_k}(\1) = 0\qquad \forall\; i,\, j,\, k \text{ distinct}
    \end{array}
    \right\}.
\end{equation}
The last condition is introduced in order to reduce the size of the semidefinite program. It demands that when expressed in terms of the variables $\tilde{x}_i=x_i-1$, every term from a polynomial in $W$ involves at most two different variables.
For $n\ge3$ we have $\dim W =\binom{n+3}{3}-\binom{n}{3}-3 = 3n(n+1)/2 - 2$, while $\dim W_\mathrm{max} = \binom{n+3}{3}-3$. When $n=17$ these dimensions are 457 and 1137 respectively.

Given this choice of $W$, we formulate \eqref{eq:cyclesosfeas} as a semidefinite feasibility problem and we use an approach very similar to that given in Section \ref{provable} to obtain an exact feasible point. The only difference is that the computations are done in the number field $\Q[\cos(2\pi/n)]$ instead of $\Q$.
If the obtained Gram matrix is strictly feasible, an $LDL^{\top}$ factorization can be performed to obtain a sum-of-squares certificate for the nonnegativity of $F(x)$, proving that $\alpha = \lambda/2$ for the $n$-cycle.
Such a proof can be readily verified in a computer algebra system, by first squaring and adding together the polynomials in the certificate, and then checking that the result is equal to $F(x)$.

\subsection{Additional practical considerations}
In this subsection, we will describe some further details of the method used to find exact proofs for the $n$-cycle.
Most importantly, we show how to use representation theory-based symmetry reduction techniques, described in \cite{Gatermann2004}, to determine a polynomial basis for $W$ with respect to which the Gram matrix in a sum of squares representation of $F_h(x)$ can be chosen to be block diagonal.
By using such a basis, we are able to find more compact certificates than we would otherwise have done.
The amount of computation required is also reduced, allowing us to search for sum-of-squares proofs for larger $n$ than would otherwise be possible.

\subsubsection{Symmetry reduction}\label{sec:symmred}
The automorphism group of the $n$-cycle is $D_{2n}$, the dihedral group of order $2n$ generated by the cyclic shift $i \in \ZZ_n \mapsto i+1$ and the reflection $i \in \ZZ_n \mapsto -i$.
It is easy to see that the function $F(x)$ is invariant under the action of permuting the indices of the $x_i$ according to elements of $D_{2n}$.
In other words, $F(x)=F(x_{\sigma})$ for each $\sigma\in D_{2n}$, where we use the notation $(x_\sigma)_i:=x_{\sigma(i)}$.
Furthermore, if \eqref{eq:cyclesosfeas} is feasible for some $h \in \R[x]_{4}$, then $h$ can always be chosen to be invariant under this action as well. This can be seen by considering $\frac{1}{2n}\sum_{\sigma\in D_{2n}}h(x_\sigma)$.

The group $D_{2n}$ naturally acts on $\R[x]$ by permuting the variables. The subspace $W$ defined in Equation \eqref{eq:cyclesubspaceW} happens to be an invariant subspace for this action.
By Maschke's theorem, $W$ can be written as a direct sum of irreducible representations of $D_{2n}$. 
Moreover, we can group each of these into isotypic components corresponding to the unique irreducible representation of $D_{2n}$ to which they are equivalent, i.e.,
\begin{equation}\label{eq:maschke}
    W = \bigoplus_{i=1}^{\frac{n+3}{2}}\left(\bigoplus_{j=1}^{m_i} W_{ij}\right).
\end{equation}
The irreducible representations of $D_{2n}$ are well known: for $n$ odd (which is the case of interest to us), $D_{2n}$ has two 1-dimensional irreducible representations and $\bigl(\frac{n-1}{2}\bigr)$ 2-dimensional irreducible representations (so $\frac{n+3}{2}$ in total). A standard analysis allows us to get the multiplicities $m_i$ with which each irreducible representation appears in $W$, and the explicit subspaces $W_{ij}$ in \eqref{eq:maschke}, see Appendix \ref{sec:Wdecompirrep}. In particular each 2-dimensional representation of $D_{2n}$ appears with multiplicity $\frac{3n+3}{2}$ or $\frac{3n+1}{2}$, the trivial 1-dimensional representation appears with multiplicity $n+2$, and the sign 1-dimensional representation appears with multiplicity $\frac{n-1}{2}$.

Let $N = \dim W = 3n(n+1)/2 - 2$ and fix a basis $b(x) = (b_1(x),\ldots,b_{N}(x))$ of $W$ adapted to the decomposition \eqref{eq:maschke}, and such that the action of $D_{2n}$ on $W$ is \emph{orthogonal}. 
By this we mean that for any $\sigma \in D_{2n}$ we can write $b(x_{\sigma}) = P_{\sigma} b(x)$ for some orthogonal matrix $P_{\sigma}$.
Our semidefinite feasibility problem \eqref{eq:cyclesosfeas} can be written as
\begin{equation}
\label{eq:sdpfeascycle}
\text{find $h \in \R[x]^{\text{inv}}_{4}$ and $Q \in \S^{N}_+$ such that } F_h(x) = b(x)^{\top} Q b(x)
\end{equation}
where $\R[x]^{\text{inv}}$ is the ring of polynomials invariant under the action of $D_{2n}$, and $F_h$ is the polynomial defined in \eqref{eq:Fhcycle}. 
Since $F_h$ is invariant under the action of $D_{2n}$, an averaging argument tells us that we can choose $Q$ to satisfy $P_{\sigma}^{\top} Q P_{\sigma} = Q$ for all $\sigma \in D_{2n}$. A standard application of Schur's lemma tells us that such $Q$ must be block diagonal, with one block for each isotypic component. 
With a little care, it can be actually ensured that each of these isotypic blocks of $Q$ is itself block diagonal, with a number of blocks equal to the dimension of the corresponding irreducible representation.
This analysis allows us to reduce the single semidefinite constraint of size $N\times N$ in \eqref{eq:sdpfeascycle}, to smaller block constraints of size $m_i \times m_i$.

\subsection{Outcome}

For the cases $n\in\{5,7,9,\dots,21\}$, we were able to find exact sum-of-squares certificates of \eqref{eq:cyclesosfeas}.
Together with the discussion at the beginning of Section \ref{sec:approach}, these certificates prove that the log-Sobolev constant of the $n$-cycle is half of its spectral gap for odd $n>3$ up to 21.
The decompositions are made available online\footnote{\proofsurl}, along with the Julia code used to generate them and the few lines of SageMath code required to verify them.
The exact computer algebra necessary to convert a numerical SDP solution into a sum-of-squares certificate of \eqref{eq:cyclesosfeas} beyond the case $n=21$ is beyond the capabilities of the computer used to perform the computations for in this paper; however, experiments carried out entirely in floating point arithmetic suggest that \eqref{eq:cyclesosfeas} continues to hold for larger $n$ (at least up to $n=35$).


\section{Discussion}\label{sec:discussion}

In this paper we have presented a computational method, based on semidefinite programming, to estimate the logarithmic Sobolev constant of a Markov kernel on a finite state space. Our method provably converges to the true value of the logarithmic Sobolev constant as the level of the hierarchy $d \to \infty$.
Moreover, we have shown how to extract rigorous certificates of the validity of these bounds.
The accuracy of the method was investigated on several example chains.
Finally, we have used this method to derive the exact logarithmic Sobolev constant for some Markov chains for which it was not previously known (to the best of our knowledge).
We hope this can be used by others as a tool to guide future exploration of logarithmic Sobolev inequalities on discrete spaces.

\paragraph{Implementation} The code necessary to reproduce the results in this paper is made available online\footnote{\repourl}.
It is written in the Julia programming language \cite{julia} and makes use of \texttt{JuMP} \cite{jump}, a popular mathematical optimization package.
\texttt{Nemo}, a computer algebra package for Julia, is used for exact computations in rational numbers and in cyclotomic fields (necessary for the results in Section \ref{sec:cycle}).

\paragraph{Further work}
There are many possible directions for future work.
\begin{itemize}
    \item Unless the Markov kernel under consideration has a large symmetry group, the method we have proposed is generally capable of yielding accurate results on a typical personal computer only for Markov chains whose state space has size $n\lesssim 13$.
    It would of course be of interest to develop a method which remains performant for chains supported on much larger state spaces, or to improve the method presented in this paper in this direction.
    \item The works \cite{boyd2004fastest,sun2006fastest} proposed a semidefinite programming-based algorithm for finding the fastest-mixing Markov process on a graph by maximizing the spectral gap, and compared the resulting optimal chains with popular heuristics such as the maximum-degree method and the Metropolis-Hastings algorithm. The problem of searching for a fastest-mixing process in terms of the logarithmic Sobolev constant was actually suggested in the former paper as a future line of research. 
    The approach described in the present paper not only allows us to lower bound the logarithmic Sobolev constant of a particular chain, but it can also be used to search for fast-mixing Markov processes in the sense of maximizing the logarithmic Sobolev constant. 
    Indeed, given a graph with vertices $\cX$ and edge set $E$, and a target distribution $\pi$, we can consider a modified version of the sum-of-squares programs considered in this paper:
    \[
        \begin{aligned}
            \min \quad \gamma \quad\st \quad& \frac{1}{2}\sum_{i,j\in\cX}\pi_i \widetilde{K}_{ij}(x_i-x_j)^2 - \sum_{i\in\cX}\pi_i p_i(x_i) \in \SOS_d(S^{\pi})\\
            & \widetilde{K}\1 = \gamma\1\\
            & \pi\widetilde{K} = \gamma\pi\\
            &\widetilde{K} \ge 0\\
            & \widetilde{K}_{ij} = 0 \;\forall (i,j)\notin E.
        \end{aligned}
    \]
    This is a sum-of-squares relaxation of the problem of maximizing the log-Sobolev constant $\gamma\inv$ over all Markov kernels $\gamma\inv\widetilde{K}$ supported on $E$ and having stationary distribution $\pi$.
    In the future, it would be interesting to compare the properties of the chains derived in \cite{boyd2004fastest,sun2006fastest} with those obtained from the optimization problem above.
\end{itemize}

\bibliography{bib}{}
\bibliographystyle{plain}

\appendix


\section{Proofs of lemmas for Proposition \ref{feas}}
\label{proofs}

\begin{lemma}
\label{lem:spanbibj}
Let $a \in \RR^n$, and let $b_1(x),\ldots,b_N(x)$ be a basis for the subspace $\{p \in \RR[x]_d : p(a) = 0\}$. Then
\begin{equation}
\label{eq:spanbibj}
\linspan \left\{ b_i(x) b_j(x) : 1\leq i,j \leq N \right\} = \left\{p \in \RR[x]_{2d} : p(a) = 0 \text{ and } \nabla p(a) = 0 \right\}.
\end{equation}
\end{lemma}
\begin{proof}
It is immediate to verify that the inclusion $\subseteq$ holds in \eqref{eq:spanbibj}. To prove the reverse inclusion, we assume, without loss of generality that $a=0$. In this case, the right-hand side corresponds to degree $2d$ polynomials with no constant or linear monomials. Thus it suffices to show that any monomial $x^{\gamma}$ where $2\leq |\gamma| \leq 2d$ is in $\linspan \left\{ b_i(x) b_j(x) : 1\leq i,j \leq N \right\}$. We can write $x^{\gamma} = x^{\alpha} x^{\beta}$ where $1\leq |\alpha|,|\beta|\leq d$, and so, by assumption $x^{\alpha} = \sum_{i} \lambda_i b_i$ and $x^{\beta} = \sum_{i} \mu_i b_i$ so that $x^{\gamma} = \sum_{ij} \lambda_i \mu_j b_i b_j$ is a linear combination of the $b_i b_j$.
\end{proof}

\begin{lemma}
\label{lem:sosdecompVarpi}
Let $b_1(x),\ldots,b_N(x)$ be a basis of the subspace of polynomials
\begin{equation}
\label{eq:subspacep1=0}
V_d = \{p \in \RR[x]_{d} : p(\1) = 0\}.
\end{equation}
Let $\pi_1,\ldots,\pi_n > 0$ such that $\sum_{i=1}^n \pi_i = 1$.
Then there exists a positive \emph{definite} matrix $A$ of size $N\times N$ such that
\begin{equation}
\label{eq:Qpdef}
\left(\sum_{d'=0}^{d-1}\norm{x}_\pi^{2d'}\right) \norm{x-\1}_\pi^2  = [b_i(x)]^T A [b_i(x)].
\end{equation}
\end{lemma}
\begin{proof}
The existence of a positive \emph{semidefinite} matrix $A$ that satisfies \eqref{eq:Qpdef} is straightforward because the left-hand side is a sum of squares.

The main point of the lemma is that one can choose $A$ to be positive definite. Note that it suffices to prove the lemma for one choice of spanning set of \eqref{eq:subspacep1=0}, say $\sigma_1,\ldots,\sigma_M$ with $M \geq N$. This is because if $b_1,\ldots,b_N$ is a basis of \eqref{eq:subspacep1=0} then we can write $[\sigma_i(x)] = R [b_i(x)]$ where $R$ is a $M\times N$ matrix with $\rank(R) = N$, and so $[\sigma_i(x)]^T A [\sigma_i(x)] = [b_i(x)]^T R^T A R [b_i(x)]$ where $R^T A R \pd 0$.

    For each $d'>0$ define the collection of polynomials indexed by $\alpha\in [n]^{d'}$,
    \[\sigma^{d'}_\alpha(x) = (\pi_{\alpha_1}\dots\pi_{\alpha_{d'}})^{\frac{1}{2}}(x_{\alpha_1}-1)x_{\alpha_2}\dots x_{\alpha_{d'}}.\]
    Clearly, we have 
    \[\norm{x-\1}_\pi^2 \norm{x}_\pi^{2d'} = \sum_{\alpha\in [n]^{d'}}\sigma^{d'}_\alpha(x)^2.\]
    To prove the lemma, we need to show that the collection of $(\sigma^{d'}_{\alpha})_{d'=1}^d$ spans the subspace $V_d = \{p \in \RR[x]_{d} : p(\1) = 0\}$.
    We prove this by induction. The base case $d=1$ is certainly true, since $\{x_1-1\dots,x_n-1\}$ is a basis for $V_1$.
    Let $p\in V_d$ for $d>1$. We can always write $p(x)=x_1p_1(x) + \dots x_np_n(x) + p_0$, where $p_0$ is a constant and $p_i\in\R[x]_{d-1}$ for $i\in[n]$. We then rewrite $p$ as
    \[p(x) = \underbrace{\sum_{i=1}^n (x_i-1)p_i(x)}_{q(x)} + \underbrace{\sum_{i=0}^n p_i(x)}_{r(x)}.\]
    Note that $r(\1)=p(\1)=0$, so we have $r\in V_{d-1}$. Using the induction hypothesis, $r\in\bigoplus_{d'=1}^{d-1} \operatorname{span}\{ \sigma^{d'}_{\alpha}\}$.
    On the other hand, $q$ is a sum of terms of the form $(x_i-1)x^\alpha$ where $x^\alpha$ is a monomial of degree at most $d-1$, i.e. $q\in\bigoplus_{d'=1}^{d} \operatorname{span} \{ \sigma^{d'}_{\alpha}\}$.
    Therefore, $p\in\bigoplus_{d'=1}^{d} \operatorname{span} \{ \sigma^{d'}_{\alpha}\}$.
\end{proof}

\section{Proof of Proposition \ref{prop:rounding} on rounding}
\label{sec:proofrounding}

\begin{proof}[Proof of Proposition \ref{prop:rounding}]
Let $\hat{Z},\,Z^*$ be the symmetric matrices corresponding to $\hat{z},\,z^*$. By assumption we know that $\lambdamin(\hat{Z}) \geq \varepsilon$. We need to prove that $\lambdamin(Z^*) \geq 0$. To do so it suffices to show that $\|\hat{Z} - Z^*\|_{\sigma} \leq \varepsilon$, where $\|\cdot \|_{\sigma}$ denotes the spectral norm.
We have the following sequence of inequalities, where $\|\cdot\|_F$ is the Frobenius norm:
\[
    \begin{aligned}
        \norm{\hat{Z} - Z^*}_\sigma &\leq \norm{\hat{Z} - Z^*}_F\\
        & \leq \sqrt{2}\norm{\begin{pmatrix}\hat{y}-y^*\\\hat{z}-z^*\end{pmatrix}}_2 \\
        &= \sqrt{2}\norm{\begin{pmatrix}\hat{y}\\\hat{z}\end{pmatrix} - \Pi_\mathcal{A}\begin{pmatrix}\hat{y}\\\hat{z}\end{pmatrix}}_2\\
        &= \sqrt{2}\norm{A^\top\left(AA^\top\right)\inv\left(A\begin{pmatrix}\hat{y}\\\hat{z}\end{pmatrix}-b\right)}_2\\
        &\leq \sqrt{2}\norm{A^\top\left(AA^\top\right)\inv}_\sigma\norm{A\begin{pmatrix}\hat{y}\\\hat{z}\end{pmatrix}-b}_2 \\
        &\leq \sqrt{\dfrac{2r}{\lambda_\mathrm{min}\left(AA^\top\right)}}\left(1+\norm{b}_\infty\right)\varepsilon_\mathrm{aff} = \varepsilon.
    \end{aligned}
\]
The last inequality follows by observing that 
\[\norm{A^\top\left(AA^\top\right)\inv}_\sigma = \sqrt{\lambda_\mathrm{max}\bigl((AA^\top)\inv AA^\top(AA^\top)\inv\bigr)}=\sqrt{\lambda_\mathrm{max}((AA^\top)^{-1})}=1/\sqrt{\lambdamin(AA^\top)}\;,\]
and that
\[\|A\begin{sm} \hat{y}\\ \hat{z}\end{sm} - b\|_{2} \le \sqrt{r}\|A\begin{sm} \hat{y}\\ \hat{z}\end{sm} - b\|_{\infty}  \le \sqrt{r}\left(1+\norm{b}_\infty\right)\varepsilon_\mathrm{aff}.\]
This completes the proof.
\end{proof}

\section{The cycle graph}
\label{sec:appcycle}

\subsection{Proof of Theorem \ref{thm:cyclesosWmax}}

If \eqref{eq:cyclesosfeas} is strictly feasible then there is $h \in \R[x]_{4}$ and a basis $q_1,\ldots,q_N$ of $W$ such that
\begin{equation}
\label{eq:Fhsos}
F_h(x) := F(x) + \left(\norm{x}^2-n\right) h(x) = \sum_{j=1}^{N} q_j(x)^2.
\end{equation}
(Indeed, one can check that the interior of $\Sigma(W)$ is the precisely the set of polynomials that can be written as $\sum_i q_i^2$ where $q_i$ forms a basis of $W$.)
We recall for convenience
\[
F(x) = \frac{2}{\lambda}\sum_{i\in\Z_n}(x_i - x_{i+1})^2 - \sum_{i=1}^np(x_i),
\]
where $\lambda = (1-\cos(2\pi/n))/2$ and $p(t) = 2(t-1)+3(t-1)^2+\frac{2}{3}(t-1)^3-\frac{1}{6}(t-1)^4+\frac{1}{15}(t-1)^5$ is the Taylor expansion of $2t^2 \log(t)$ to order 5.

Observe that $F_h(\1) = 0$, and so this implies that $q_j(\1) = 0$ for all $j=1,\ldots,N$. If we differentiate the identity \eqref{eq:Fhsos} we get $\nabla F_h(\1) = 0$ i.e., $- 2\1 + 2h(\1) \1 = 0$ which implies $h(\1) = 1$. Furthermore
\[
\nabla^2 F_h(\1) = \frac{4}{\lambda}(I-K) - 6I + 2 h(\1)I + \nabla h(\1) \1^{\top} + \1 \nabla h(\1)^{\top}
\]
where $K$ is the transition matrix of the simple walk on the $n$-cycle. The vectors $\phi = (\cos(2\pi i/n))_{1\leq i \leq n}$ and $\psi = (\sin(2\pi i/n))_{1\leq i \leq n}$ are eigenvectors of $(I-K)$ with eigenvalue $\lambda$, and are both orthogonal to $\1$, so we have that $\phi^\top \nabla^2 F_h(\1)\phi=\psi^\top \nabla^2 F_h(\1)\psi=0$. Using Lemma \ref{lemma:soshessian} (to follow) this implies that $\phi^{\top} \nabla q_j(\1) = \psi^{\top} \nabla q_j(\1) = 0$.

We have shown that the $q_j$ must all satisfy the linear relations
\[
q_j(\1) = 0, \; \phi^{\top} \nabla q_j(\1) = \psi^{\top} \nabla q_j(\1) = 0 \;\; \forall j=1,\ldots,N.
\]
Since by assumption the $(q_j)$ span the subspace $W$ we get the desired inclusion \eqref{eq:maximal_subspace}.

It remains to prove the lemma below.

\begin{lemma}
\label{lemma:soshessian}
    Let $F(x)\in\R[x]_{2d}$ with $F(\1)=0$ and suppose that $F$ is a sum of squares of polynomials:
    \[F(x)=\sum_{j=1}^N q_j(x)^2.\]
    Let $\nabla^2F(x)$ be the Hessian matrix of second derivatives of $F$.
    Then for any $v\in\ker \nabla^2 F(\1)$, $1\le j\le N$ we have 
    \[v^\top\nabla q_j(\1) = 0.\]
\end{lemma}
\begin{proof}
    For a real variable $t$ define the function $f(t)=F(\1+tv)$.
    One the one hand, 
    \[f''(0) = \frac{d}{dt}\left(v^\top\nabla F(\1+tv)\right)\vert_{t=0} = v^\top \nabla^2F(\1)v = 0.\]
    On the other,
    \[ 
    \begin{aligned}
        f''(0) &= \sum_{j=1}^N\frac{d^2}{dt^2}q_j^2(\1+tv)\vert_{t=0} \\
        &= 2\sum_{j=1}^N\left(\left(\frac{d}{dt}q_j(\1+tv)\vert_{t=0}\right)^2 + q_j(\1)\frac{d^2}{dt^2}q_j(\1+tv)\vert_{t=0}\right)\\
        &=2\sum_{j=1}^N\left(\frac{d}{dt}q_j(\1+tv)\vert_{t=0}\right)^2\\
        &=2\sum_{j=1}^N\left(v^\top\nabla q_j(\1)\right)^2.
    \end{aligned}
    \]
    This implies $v^{\top} \nabla q_j(\1) = 0$ for all $1\leq j\leq N$ as desired.
\end{proof}
This completes the proof of Theorem \ref{thm:cyclesosWmax}.

\subsection{Decomposition of the subspace $W$ into irreducibles}
\label{sec:Wdecompirrep}

In order to decompose $W$ (defined in \eqref{eq:cyclesubspaceW}) into irreducible representations, it will be helpful to first consider an enlarged space $\widetilde{W}$.
This space is defined as $\widetilde{W}=\{q\in\R[x]_{3}\mid\frac{d^3q}{dx_idx_jdx_k}(\1)=0\quad\forall\; i,j,k\text{ distinct}\}$.
Note that $W=\widetilde{W}\cap W_\mathrm{max}$.
In the rest of this section, we will work in the more convenient variables $\tilde{x}_i:=x_i-1$. From this point of view, $\widetilde{W}$ has a nice interpretation as the space of degree 3 polynomials, no term of which involves 3 different variables.

There is an obvious way to coarsely decompose $\widetilde{W}$ into invariant subspaces:
\begin{enumerate}[(1)]
    \item $\operatorname{span}\{1\}$, the space of constant polynomials,
    \item $\operatorname{span}\{\tilde{x}_1, \dots, \tilde{x}_n\}$, $\operatorname{span}\{\tilde{x}_1^2, \dots, \tilde{x}_n^2\}$, and $\operatorname{span}\{\tilde{x}_1^3, \dots, \tilde{x}_n^3\}$ are invariant subspaces which each induce a copy of the permutation representation.
    Moreover, for  $\Delta\in\{1,\dots, \frac{n-1}{2}\}$, the spaces \[
        \operatorname{span}\{\tilde{x}_i\tilde{x}_j\mid i,j\in\{1,\dots,n\},\; d_n(i,\,j)=\Delta\}\] are also invariant subspaces equivalent to the permutation representation via the isomorphism
        \[\tilde{x}_i\tilde{x}_j\mapsto\begin{cases} \bm{e}_{i+\frac{\Delta}{2}\pmod n}&\text{if }\Delta\text{ even}\\ \bm{e}_{i+\frac{\Delta+n}{2}\pmod n}&\text{if }\Delta\text{ odd}.\end{cases}\]
    Here $d_n(i,j):=\min\{\vert i-k\vert\st k\cong j\pmod n\}$.
    \item For $\Delta\in\{1,\dots, \frac{n-1}{2}\}$, the spaces 
    \[\operatorname{span}\{\tilde{x}_i^2\tilde{x}_j\mid i, j\in\{1,\dots,n\},\; d_n(i,\,j)=\Delta\}\] 
    are invariant subspaces which each induce a copy of the regular representation.
    Indeed, writing $D_{2n}=\langle r, s \mid r^n = s^2 = \id,\, srs=r\inv\rangle$, the map
    \[\tilde{x}^2_i\tilde{x}_j\mapsto\begin{cases}\bm{e}_{r^j}&j\equiv i+\Delta\pmod n\\ \bm{e}_{r^js} & j\equiv i - \Delta\pmod n\end{cases}\]
    is an isomorphism which respects the action of $D_{2n}$.
\end{enumerate}
In order to find a full decomposition of $\widetilde{W}$ into irreducible representations as in \eqref{eq:maschke}, it remains only to decompose the regular and permutation representations of $D_{2n}$ into irreducibles.
We omit the details (this is standard material in representation theory) except to say that the resulting symmetry-adapted basis $\tilde{b}(\tilde{x})$ for $\widetilde{W}$ can be chosen so that each basis element is a polynomial whose coefficients lie in $\Q[\cos(2\pi/n)]$.
It may be necessary to rescale certain basis elements by a constant factor outside of this field, e.g., the coefficients of $\sum_i\psi_i\tilde{x}_i = \sum_i\sin(2\pi i/n)\tilde{x}_i$ are not in $\Q[\cos(2\pi/n)]$, but those of $\sin(2\pi/n)\cdot(\sum_i\psi_i\tilde{x}_i)$ are.

We now aim to obtain a decomposition of $W$ into irreducible representations.
It is easily verified that $\operatorname{span}\{\sum_i\phi_i\tilde{x}_i,\;\sum_i\psi_i\tilde{x}_i\}$ is an invariant subspace of $\operatorname{span}\{\tilde{x}_1, \dots, \tilde{x}_n\}$.
In fact, the span of these two polynomials is irreducible, and not equivalent to any other subrepresentation of $\operatorname{span}\{\tilde{x}_1, \dots, \tilde{x}_n\}$.
It follows that both $\sum_i\phi_i\tilde{x}_i$ and $\sum_i\psi_i\tilde{x}_i$ are elements of $\tilde{b}(\tilde{x})$, the basis of $\widetilde{W}$ obtained from the procedure sketched out above.
Therefore a symmetry-adaped basis for $W=\widetilde{W}\cap{W_\mathrm{max}}$ can be obtained by dropping from $\tilde{b}(\tilde{x})$ the polynomials $\{\sum_i\phi_i\tilde{x}_i$ and $\sum_i\psi_i\tilde{x}_i\}$, as well as the constant polynomial.

\end{document}